\documentclass[11pt]{amsart}

\usepackage{epigamath}

\usepackage[english]{babel}

\numberwithin{equation}{section}

\usepackage{dsfont}
\usepackage[T1]{fontenc}
\usepackage{amssymb,url,xspace}
\usepackage{mathrsfs}  
\usepackage{hyperref}
\usepackage{url}
\usepackage{cite,amsmath,bm} 
\usepackage{paralist}
\usepackage{datetime} 
\usepackage{colortbl}
\usepackage[dvipsnames]{xcolor}
\usepackage{mathtools}
\usepackage{euscript}
\usepackage[all]{xy}

\usepackage{comment}  
\usepackage{cite}  
\usepackage{thmtools}
\usepackage[shortlabels]{enumitem}
\usepackage{letltxmacro} 
\usepackage{nameref}
\usepackage{cleveref}
\usepackage{calc}
\usepackage{interval}
\usepackage{manfnt}
\usepackage{tikz-cd}

\usepackage{faktor}

\theoremstyle{plain}
\newtheorem{thmx}{Theorem}
\renewcommand{\thethmx}{\Alph{thmx}} 
\newtheorem{thm}{Theorem}[section]  
\newtheorem{lem}[thm]{Lemma}

\newtheorem{proposition}[thm]{Proposition}
\newtheorem{cor}[thm]{Corollary}
\newtheorem{corx}[thmx]{Corollary}
\newtheorem{claim}[thm]{Claim}

\theoremstyle{definition}
\newtheorem{dfn}[thm]{Definition}

\theoremstyle{remark}
\newtheorem{rem}[thm]{Remark} 

\crefname{lem}{Lemma}{Lemmas} 
\crefname{conjecture}{Conjecture}{Conjectures}
\crefname{thm}{Theorem}{Theorems}
\crefname{proposition}{Proposition}{Propositions}
\crefname{dfn}{Definition}{Definitions}
\crefname{rem}{Remark}{Remarks}
\crefname{cor}{Corollary}{Corollaries}
\crefname{corx}{Corollary}{Corollaries}
\crefname{problem}{Problem}{Problems}
\crefname{thmx}{Theorem}{Theorems}
\crefname{claim}{Claim}{Claims}
\crefname{assumption}{Assumption}{Assumptions}
\crefname{main}{Main Theorem}{Main Theorems}

\theoremstyle{plain}

\newlist{thmlist}{enumerate}{1}
\setlist[thmlist]{label={\rm(\roman*)}, ref=\thethm{\rm(\roman{thmlisti})}}

\addtotheorempostheadhook[thm]{\crefalias{thmlisti}{thm}}
\addtotheorempostheadhook[assumpsion]{\crefalias{thmlisti}{assumption}}
\addtotheorempostheadhook[proposition]{\crefalias{thmlisti}{proposition}}
\addtotheorempostheadhook[dfn]{\crefalias{thmlisti}{dfn}}
\addtotheorempostheadhook[lem]{\crefalias{thmlisti}{lem}}
\addtotheorempostheadhook[main]{\crefalias{thmlisti}{main}}
\addtotheorempostheadhook[rem]{\crefalias{thmlisti}{rem}}

\newlist{thmenum}{enumerate}{1} 
\setlist[thmenum]{label={\rm(\roman*)}, ref=\thethmx\rm{(\roman{thmenumi})}}
\crefalias{thmenumi}{thmx} 

\newlist{corlist}{enumerate}{1} 
\setlist[corlist]{label={\rm(\roman*)}, ref=\thecorx\rm{(\roman{corlisti})}}
\crefalias{corlisti}{corx} 

\setlist[enumerate,1]{label={\rm(\arabic*)}, ref=({\rm\arabic*})}

\makeatletter
\newsavebox{\@brx}
\newcommand{\llangle}[1][]{\savebox{\@brx}{\(\m@th{#1\langle}\)}%
	\mathopen{\copy\@brx\kern-0.5\wd\@brx\usebox{\@brx}}}
\newcommand{\rrangle}[1][]{\savebox{\@brx}{\(\m@th{#1\rangle}\)}%
	\mathclose{\copy\@brx\kern-0.5\wd\@brx\usebox{\@brx}}}
\makeatother

\def\ep{\varepsilon}

\def\p{p}

\def\rank{\operatorname{rank}}

\def\uni{\operatorname{uni}}
\def\tr{\operatorname{tr}}
\def\Im{\operatorname{Im}}
\def\hod{\operatorname{hod}}
\def\Sym{\operatorname{Sym}}
\def\deg{\operatorname{deg}}
\def\ord{\operatorname{ord}}

\newcommand{\cD}{\mathcal D}

\newcommand{\cU}{\mathcal U}
\newcommand{\cW}{\mathcal W}
\newcommand{\cO}{\mathcal O}

\newcommand{\kL}{\mathfrak L} 

\newcommand{\bC}{\mathbb{C}}

\newcommand{\bP}{\mathbb{P}}

\newcommand{\bZ}{\mathbb{Z}}

\def\db{\bar{\partial}}

\def\d{\partial} 
\def\hess{{\rm d}{\rm d}^{\rm c}}

\def\lra{\longrightarrow}

\def\sn{\sqrt{-1}}

\def\End{\operatorname{End}}

\newcommand{\fa}{\mathbf{a}}
\newcommand{\fb}{\mathbf{b}} 

\newcommand{\diae}{{}^\diamond\! E}
\newcommand{\dia}{{}^\diamond\!}
\newcommand{\pae}{\mathcal{P}_{\!\bm{a}}E}
\newcommand{\pf}{\mathcal{P}_{\!\bm{1}}}
\newcommand{\pahe}{\mathcal{P}^h_{\!\bm{a}}E}
\newcommand{\pbe}{\mathcal{P}_{\!\bm{b}}E}
\newcommand{\shb}{(E=\oplus_{p+q=m}E^{p,q},\theta,h)}

\newcommand{\supth}[1]{\ensuremath{#1^{\mathrm{th}}}}

\EpigaVolumeYear{7}{2023} \EpigaArticleNr{12} \ReceivedOn{August 26, 2021}
\InFinalFormOn{September 17, 2022}
\AcceptedOn{January 14, 2023}

\title{Big Picard theorems and algebraic hyperbolicity   for  varieties  admitting a variation of Hodge structures}
\titlemark{Hyperbolicity of varieties admitting $\boldsymbol{\bC}$-PVHS}

\author{Ya Deng} 
\address{CNRS, Institut \'Elie Cartan de Lorraine, Universit\'e de Lorraine, F-54000 Nancy, France}
\email{ya.deng@math.cnrs.fr}

\authormark{Y.~Deng}

\AbstractInEnglish{In this paper, we study  various  hyperbolicity properties for a quasi-compact K\"ahler manifold $U$  which admits a complex polarized variation of Hodge structures  so that each fiber of the period map is zero-dimensional.   In the first part, we prove that   $U$ is algebraically hyperbolic and that the generalized big Picard theorem holds for $U$.    In the second part, we  prove that there is a finite \'etale cover $\tilde{U}$ of $U$ from a  quasi-projective manifold $\tilde{U}$ such that any  projective compactification $X$  of $\tilde{U}$ is Picard hyperbolic modulo the boundary $X-\tilde{U}$, and any irreducible subvariety of $X$ not contained in $X-\tilde{U}$ is of general type.
  This result coarsely incorporates  previous works  by Nadel, Rousseau, Brunebarbe and Cadorel on the hyperbolicity of compactifications of quotients of bounded symmetric domains by torsion-free lattices.} 


\acknowledgement{This work is supported by  ``{le Fonds Chern}'' during the author's visit at the IH\'ES.}

\begin{document}


\maketitle

\begin{prelims}

\DisplayAbstractInEnglish

\bigskip

\DisplayKeyWords

\medskip

\DisplayMSCclass

\end{prelims}


\newpage

\setcounter{tocdepth}{1}

\tableofcontents


\section{Introduction} 
\subsection{Background}
The classical big Picard theorem says that any holomorphic map  from the punctured disk $\Delta^*$  into $\bP^1$
which omits three points can be extended to a
holomorphic map $\Delta\to \bP^1$, where $\Delta$ denotes the unit disk.
Therefore, we introduce a new notion of hyperbolicity which generalizes the big Picard theorem. We say a complex manifold $U$ is \emph{quasi-compact K\"ahler} if it is a Zariski open subset of a compact K\"ahler manifold $Y$.  Such a $Y$ will be called a  \emph{smooth K\"ahler compactification} of $U$.

\begin{dfn}[Picard hyperbolicity]\label{def:Picard}
A  quasi-compact K\"ahler manifold $U$ is called \emph{Picard hyperbolic}    if there is a smooth K\"ahler  compactification $Y$ of  $U$ such that any holomorphic map $f\colon\Delta^*\to U$ extends to a holomorphic map $\bar{f}\colon\Delta\to Y$.
\end{dfn}

We will prove in \cref{lem:well-def} that this definition does not depend on the compactification of $U$.   Picard hyperbolic varieties first attracted the author's interest because of the recent interesting work \cite{JK18} by   Javanpeykar--Kucharczyk   on the \emph{algebraicity of analytic maps}. In \cite[Definition 1.1]{JK18}, they introduce a new notion of hyperbolicity: a reduced quasi-projective variety  $U$ is \emph{Borel hyperbolic} if  any holomorphic map from a quasi-projective variety to $U$ is   \emph{algebraic}.  In  \cite[Corollary 3.11]{JK18}, they prove  that a Picard hyperbolic variety is    {Borel hyperbolic}. We refer the readers to \cite[Section~1]{JK18}
for their motivation on the Borel hyperbolicity.  Picard hyperbolic varieties fascinate us further when we realize in  \Cref{extension theorem} that  a more general extension theorem is also valid for them: any holomorphic map from $\Delta^p\times (\Delta^*)^q$ to the manifold $U$ in \Cref{def:Picard} extends to a \emph{meromorphic map} from $\Delta^{p+q}$ to $Y$. 

By Borel \cite{Bor72} and Kobayashi--Ochiai \cite{KO71}, it has long been known to us that the quotients of bounded symmetric domains by torsion-free arithmetic lattices are hyperbolically embedded into their Baily--Borel compactification, and thus they are Picard hyperbolic (see \cite[Theorem 6.1.3]{Kob98}).  Period domains, first introduced by Griffiths \cite{Gri68} and   later systematically studied  by him in the seminal work \cite{Gri68b,Gri70,Gri70b}, are classifying spaces of polarized Hodge structures and are  a generalization of bounded symmetric domains. 
  In \cite[Section 1.1]{JK18},  
  Javanpeykar--Kucharczyk conjectured that an algebraic variety $U$ which admits an \emph{integral variation of Hodge structures} ($\bZ$-VHS for short) with quasi-finite period map is Borel hyperbolic. 
  Their conjecture was recently proved in   a  remarkable work of  Bakker--Brunebarbe--Tsimerman \cite{BBT18}.   The  proof is based on the tame geometry for quotients $\faktor{\cD}{\Gamma}$  of period domains $\mathcal{D}$ by arithmetic groups $\Gamma$ containing the monodromy group of the $\bZ$-VHS. However, when one studies Picard hyperbolicity (or Borel hyperbolicity) for varieties admitting the more general \emph{complex polarized variation of Hodge structures} ($\bC$-PVHS for short), there are several problems which seem difficult to tackle if one uses o-minimal geometry in \cite{BBT18}:
\begin{itemize}
	\item The monodromy group $\Gamma$ might   act non-discretely on the period domain $\cD$, and thus the quotient of the period domain by the monodromy group $\faktor{\cD}{\Gamma}$ is even non-Hausdorff.
	\item The local monodromies at infinity might not be quasi-unipotent, though it always is  for a $\bZ$-VHS by the theorem of Borel. 
\end{itemize}
 Therefore, the great differences between a $\bZ$-VHS and a $\bC$-PVHS require   new ideas if one would like to extend the theorem by Bakker--Brunebarbe--Tsimerman to the context of $\bC$-PVHSs.  This is one of the main goals in this paper. 
 In the first part, we prove the Picard hyperbolicity of quasi-compact K\"ahler manifolds admitting a  $\bC$-PVHS using techniques in complex-analytic geometry and Hodge theory.

\subsection{Big Picard theorem and algebraic hyperbolicity}  

\begin{thmx}\label{main}
  Let $U$ be a quasi-compact K\"ahler manifold.  Assume that there is a $\bC$-PVHS over $U$  so that  each fiber of the period map $U_{\uni}\to \mathcal{D}$ is zero-dimensional, where $U_{\uni}$ is the universal cover of $U$ and $\mathcal{D}$ is the period domain associated to the above $\bC$-PVHS.  Then    $U$ is both algebraically hyperbolic and Picard hyperbolic. In particular, $U$ is Borel hyperbolic.
\end{thmx}

Note that we make no assumptions on the local monodromies of the $\bC$-PVHS at infinity (which can be non--quasi-unipotent) or on its global monodromy group (thus it might act non-discretely on $U$). 
Let us mention that when the $\bC$-PVHS over $U$ in \cref{main} is moreover a $\bZ$-PVHS, the Borel hyperbolicity of  $U$ in \cref{main} has been proven in \cite[Corollary 7.1]{BBT18}, and the algebraic hyperbolicity of   $U$  is implicitly shown by Javanpeykar--Litt in \cite[Theorem 4.2]{JL19} if the local monodromies of the $\bZ$-VHS at infinity are unipotent (see \Cref{JL}).  Our proof  of \cref{main} is based on complex-analytic and Hodge-theoretic methods, and it does not use the delicate o-minimal geometry in \cite{PS08,PS09,BKT18,BBT18}. Indeed, we are not sure that the tame geometry for period domains of $\bZ$-PVHSs in \cite{BKT18,BBT18} can be applied to prove  \cref{main}.

We can even generalize \cref{main} to higher-dimensional domain spaces.

\begin{cor}[= \cref{main} + \Cref{extension theorem}]\label{extension}
Let $U$ be the quasi-compact K\"ahler manifold   in \cref{main}, and let $Y$ be a smooth K\"ahler compactification of\, $U$. Then  any  holomorphic map $f\colon\Delta^p\times (\Delta^*)^q\to U$ extends to a meromorphic map $\overline{f}\colon\Delta^{p+q}\dashrightarrow Y$. In particular, if    $W$ is a Zariski open subset of a compact complex manifold $X$, then any holomorphic map  $g\colon W\to U$ extends to a meromorphic map $\overline{g}\colon X\dashrightarrow Y$.
\end{cor}

\subsection{Hyperbolicity for the compactification after a finite \'etale cover}
The second main result of this paper is on the hyperbolicity for the compactification of some  finite \'etale cover of the quasi-compact K\"ahler manifold $U$ in \cref{main}. Let us first introduce several definitions of hyperbolicity. We refer the readers to the recent survey by Javanpeykar \cite[Section 8]{Jav20} for more   conjectural relations  among them.  

\begin{dfn}[Notions of hyperbolicity]\label{def:notion}
Let $(X,\omega)$ be a compact K\"ahler manifold, and let $Z\subsetneq X$ be a   closed subset of $X$. 
\begin{thmlist}\label{notions}
	\item  The manifold $X$ is   \emph{Kobayashi hyperbolic modulo $Z$} if the  {Kobayashi pseudo distance} satisfies $d_{X}(x,y)>0$ for any distinct points $x,y\in X$ not both contained in $Z$.  
	\item   The manifold $X$ is   \emph{Picard hyperbolic modulo $Z$} if   any holomorphic map $\gamma\colon \Delta^*\to X$ whose image is not contained in $Z$ extends across the origin.   
	\item The manifold $X$ is   \emph{Brody hyperbolic modulo $Z$} if  any entire curve  $\gamma\colon\bC\to X$ is  contained in $Z$.   
		\item  \label{def:algebraic} The manifold $X$ is   \emph{algebraically hyperbolic modulo $Z$} if there is an $\ep>0$ so that for any compact irreducible   curve $C\subset X$ not contained in $Z$, one has
		$$
		2g(\tilde{C})-2\geq \ep \deg_{\omega}C,
		$$
		where $g(\tilde{C})$ is the genus of the  normalization $\tilde{C}$ of $C$, and $\deg_{\omega}C:=\int_{C}\omega$.
\end{thmlist}  
\end{dfn}

Note that \cref{def:algebraic} was first introduced in \cite{JX22}. 
It is easy to show that if $X$ is   Kobayashi hyperbolic  modulo  $Z$ (resp.\ Picard hyperbolic  modulo  $Z$), then $X$ is Brody hyperbolic modulo $Z$.

The second main result  of the paper is the following theorem.

\begin{thmx}\label{thmx:refined}
Let $U$ be a quasi-compact K\"ahler manifold.   Assume that there is a  $\bC$-PVHS  over $U$  so that  each fiber of the period map   is zero-dimensional.     Then there are a   quasi-projective manifold $\tilde{U}$ and a finite \'etale cover $\tilde{U}\to U$   such that for any projective compactification $X$ of\, $\tilde{U}$,
	\begin{thmenum}
		\item   \label{general type2}  an irreducible Zariski closed  subvariety of $X$  not contained in $\tilde{D}:=X-\tilde{U}$   is  of  general type;  
		\item  \label{Picard} the variety $X$  is   Picard hyperbolic modulo $\tilde{D}$;
		\item  \label{Brody} the variety $X$  is  Brody hyperbolic modulo $\tilde{D}$;
		\item  \label{Algebraic} the variety $X$  is  algebraically hyperbolic modulo $\tilde{D}$.
	\end{thmenum}
\end{thmx}

By the work of Deligne (see \cite[Theorem 7.10]{Mil13}),  the quotient   of any  bounded symmetric domain   by a torsion-free lattice  always admits a  $\bC$-PVHS whose period map is \emph{immersive everywhere}. \cref{thmx:refined} then yields the following. 

\begin{corx}\label{Picard symmetric}
	Let $U$   be the quotient of   a  bounded symmetric domain  by  a torsion-free   lattice. Then    there is a   finite \'etale cover $\tilde{U}\to U$ from a quasi-projective manifold $\tilde{U}$ with  any projective compactification $X$ of\, $\tilde{U}$    Picard and algebraically hyperbolic modulo $X-\tilde{U}$. 
\end{corx} 

Let us stress here that   Nadel \cite{Nad89} and Rousseau \cite{Rou16} proved that the variety $X$ in \cref{Picard symmetric} is Brody and Kobayashi hyperbolic modulo $X-\tilde{U}$, and Brunebarbe \cite{Bru16} and Cadorel \cite{Cad16,Cad18} proved that any Zariski closed subvariety not contained in $X-\tilde{U}$ is of general type.  \cref{thmx:refined} thus incorporates their results, but at the cost of loss of effectivity for the level structures (see \cref{rem:compare})  due to the generality of our result in \cref{thmx:refined}.

\subsection{Main strategy}

\subsubsection{Negatively curved Finsler metric}\label{sec:describe}
 
Let $Y$ be  a compact K\"ahler manifold, and let $D$ be a simple normal crossing divisor on $Y$.  Assume that there is a $\bC$-PVHS   over $U:=Y-D$.   In \cite{Gri70},  Griffiths constructed a metrized   line bundle on $U$ whose curvature is semipositive and strictly positive at the points where the period map is  immersive.  Based on the work by Simpson and Mochizuki, in \cref{prop:bigness}, we can extend this Griffiths line bundle over $Y$ to obtain  a more refined positivity result.\footnote{If the local monodromy around $D$ is unipotent, this is well known.}
We then construct a special  system of log Hodge bundles $(E,\theta)=(\oplus_{p+q=m} E^{p,q}, \oplus_{p+q=m}\theta_{p,q})$ on the log pair $(Y,D)$ so that  some higher-stage $E^{p_0,q_0}$ contains a big line bundle which admits enough \emph{local positivity along $D$}.   Inspired by our previous work \cite{Den18} on the proof of   Viehweg--Zuo's conjecture on  Brody hyperbolicity of moduli of polarized manifolds, in \cref{Deng} we show  that $(E,\theta)$ still enjoys a  ``{partially}'' infinitesimal Torelli property.  These results enable us to construct a negatively curved and generically positive definite Finsler metric on $T_Y(-\log D)$  in a similar vein as \cite{Den18}.

\begin{thm}[ = \cref{prop:new} + \cref{thm:uniform}]\label{thmx:uniform}
	Let $ Y$ be a compact K\"ahler manifold,   and let $D$ be a simple normal crossing divisor on $Y$.  Assume that there is a     $\bC$-PVHS  over $U:=Y-D$ whose period map is   immersive   at one point.  Then there are a \emph{Finsler metric} $h$ $($see \cref{def:Finsler}\,$)$ on $T_Y(-\log D)$ which is positive definite on a dense Zariski open subset $U^\circ$\, of $Y- D$ and a smooth K\"ahler form $\omega$ on $Y$ such that   for any holomorphic map $\gamma\colon C\to U$  from an open set $C\subset \bC$ to $U$, one has
	\begin{align} \label{eq:information}
	 \hess\log |\gamma'(t)|_{h}^2\geq \gamma^*\omega
	\end{align}   
	when $\gamma(C)\cap U^\circ\neq\varnothing$.
\end{thm}

Let us mention that, though we only construct the (possibly degenerate) Finsler metric over $T_Y(-\log D)$, it follows from \eqref{eq:information} that we know exactly the behavior of its curvature near the boundary $D$ since $\omega$ is a smooth K\"ahler form over $Y$. The proof of \cref{main} is then based on \cref{thmx:uniform} and some criteria for the big Picard theorem established  in \cite{DLSZ} (see \cref{thm:criteria}).  
Let us also mention that the Finsler metric constructed in \cref{thmx:uniform} is also crucially used in the proof of \cref{thmx:refined}.

\subsubsection{On the hyperbolicity of the compactification}
The proof of \cref{thmx:refined} is based on \cref{strong}, whose proof is technically involved. It is worthwhile to mention that our proof is quite different from those in \cite{Nad89,Rou16,Bru16,Cad18}. All these proofs relied heavily on the special property of quotients of bounded symmetric domains by torsion-free lattices. They all applied  the toroidal compactifications by Mumford to find the  desired finite \'etale cover $\tilde{U}\to U$ when $U$ is a quotient of a bounded symmetric domain  by a torsion-free lattice. We construct the cover $\tilde{U}\to U$ in \cref{thmx:refined} in a subtle way using the residual  finiteness of the global monodromy group.   We refer the readers to the beginning of \Cref{sec:strong} for the general strategy.

\subsection{Some new developments}
Shortly after this paper appeared on arXiv,   Brunebarbe--Brotbek \cite[Theorem 1.5]{BB20} proved the Borel hyperbolicity of $U$ in \cref{main} under the additional assumption that the local monodromy of the $\bC$-PVHS at infinity is unipotent.  Moreover, they also obtained a weaker  result than \cref{Picard}  in \cite[Theorem~1.7]{BB20}, in which they showed that for a quasi-projective manifold $U$ admitting a  $\bZ$-PVHS with quasi-finite period map,  there is a finite \'etale cover $\tilde{U}\to U$ so that the projective compactification $X$ of $\tilde{U}$ is \emph{Borel hyperbolic} modulo $X-\tilde{U}$.    Our proofs are indeed quite different: Brotbek--Brunebarbe's proof is based on their Second Main Theorem using the \emph{Griffiths--Schmid metric}, which  coincides with the curvature form of the Griffiths line bundle.  
 Let us also mention that result similar to  \cref{general type2} is also obtained by Brunebarbe in \cite[Theorem 1.1]{Bru20} when the underlying local system of the $\bC$-PVHS is defined over $\bZ$.  
 
In  \cite{BD21}   Cadorel and the author generalized \cref{main,thmx:refined,Picard symmetric} in this paper to quasi-compact K\"ahler manifolds admitting \emph{nilpotent Higgs bundles}.  More recently,  in \cite[Theorem 0.1]{CDY22} Cadorel, Yamanoi and the author proved that for any complex quasi-projective  normal variety  $X$, if there is a big representation $\varrho\colon \pi_1(X)\to GL_N(\bC)$ such that the Zariski closure of $\varrho(\pi_1(X))$ is a semisimple algebraic group, then there is a proper Zariski closed subset $Z\subsetneq X$ such that 
  \begin{itemize}
  	\item  any closed subvariety of $X$ not contained in $X$ is log general type;
  	\item  $X$ is Picard hyperbolic modulo $Z$. 
  \end{itemize}
We stress here that  \cref{main} in this paper is applied in \cite{CDY22}.   

\subsection*{Acknowledgements}
 I would like to thank   Junyan Cao, Jiaming Chen, Jean-Pierre Demailly, Philippe Eyssidieux,   Bruno Klingler, Steven Lu, Mihai P\u{a}un, Erwan Rousseau, Emmanuel Ullmo  and Kang Zuo  for their interest and discussions. I especially thank    Ariyan Javanpeykar for    various comments and thank Gregorio Baldi and Beno\^it Cadorel for   fruitful discussions during the revision of this paper.   Last but not least, I am grateful to the referees for their  careful reading and very helpful comments to improve this manuscript.
 
 \section*{Notation and Conventions}
 \begin{itemize}
 	\item 	A \emph{log pair} $(Y,D)$ consists of a (possibly non-compact) complex manifold and a simple normal crossing divisor $D$. It will be called a \emph{compact K\"ahler log pair} (resp.\ \emph{projective log pair}) if $Y$ is a compact K\"ahler (resp.\ projective) manifold. 
 	\item A complex manifold is called \emph{quasi-compact K\"ahler} if it is a Zariski open subset of a compact K\"ahler manifold.
 	\item A \emph{log morphism} $f\colon(X,\tilde{D})\to (Y,D)$ between log pairs is a morphism $f\colon X\to Y$ with $\tilde{D}\subset f^{-1}(D)$.
 	\item For a big line bundle $L$ on a projective manifold, $\mathbf{B}_+(L)$ denotes its \emph{augmented base locus}  (see \cite[Definition 10.3.2]{Laz04II}).
 \end{itemize}
 
\section{Preliminaries on Hodge theory} 

\subsection{Systems of Hodge bundles} \label{sec:VHS} 
Following Simpson \cite{Sim88}, a complex polarized variation of Hodge structures ($\bC$-PVHS) is equivalent to a system of Hodge bundles. Let us recall the definition in this subsection.

\begin{dfn}[Higgs bundle]
 		A \emph{Higgs bundle} on a complex manifold $Y$ is a pair $(E,\theta)$ consisting of a holomorphic vector bundle $E$ on $Y$ and an $\cO_Y$-linear map
 		$$
 		\theta\colon E\lra E\otimes \Omega_Y^1
 		$$
 		so that $\theta\wedge\theta=0$. Such a map $\theta$ is called the \emph{Higgs field}.
\end{dfn}
        
\begin{dfn}[Harmonic bundle]
 	A \emph{harmonic bundle}  $(E,\theta,h)$ consists of a  Higgs bundle $(E,\theta)$ and  a hermitian metric $h$ for $E$ such that
 	$$
 	D:=\d_h+\db_E+\theta+\theta_h^*
 	$$
 	is flat. Here $\d_h+\db_E$ is the Chern connection, and $\theta_h^*\in C^{\infty}(Y,\End(E)\otimes \Omega^{0,1}_Y)$ is the adjoint of $\theta$ with respect to $h$.
 \end{dfn}
 
\begin{dfn}[System of Hodge bundles]\label{def:Hodge}
A \emph{system of Hodge bundles} of weight $m$ is a harmonic bundle   $(E,\theta,h)$ satisfying the following: \begin{itemize}
	\item The vector bundle $E=\oplus_{p+q=m}E^{p,q}$ is a direct sum of holomorphic vector bundles $E^{p,q}$.
	\item The map $\theta$ restricts to  $$
	\theta|_{E^{p,q}}\colon E^{p,q}\lra E^{p-1,q+1}\otimes \Omega_Y^1. 
	$$
	\item The splitting $E=\oplus_{p+q=m}E^{p,q}$  is orthogonal with respect to $h$.
\end{itemize} 
 We write $h_{p,q}=h|_{E^{p,q}}$ and $\theta_{p,q}=\theta|_{E^{p,q}}$. This harmonic metric $h$ will be called the \emph{Hodge metric}.
 \end{dfn}
 
Throughout this paper, we observe the convention that $0\leq p,q\leq m$ for the decomposition $E=\oplus_{p+q=m}E^{p,q}$. This can always be achieved if we make a Tate twist $(k,k)$ to increase the weight by $2k$ when $k\in \bZ_{>0}$ is large enough.

\subsection{Filtered bundles and parabolic Higgs bundles}\label{sec:parahiggs}
In this section, we recall the notions of filtered bundles and parabolic Higgs bundles from \cite{Sim88,Moc07}. Let $(Y,D=\sum_{i=1}^{c}D_i)$ be a log pair.

\begin{dfn}\label{def:filter}
	A filtered bundle $(E,\pae)$ on $(Y, D)$ is a locally free
sheaf $E$ on $U:=Y-D$, together with an $\mathbb{R}^c$-indexed
	filtration $\pae$  by locally free sheaves on $Y$ such that
	\begin{thmlist}
		\item $\bm{a}\in \mathbb{R}^c$ and $\pae|_U=E$; 
\item	$\pae\subset \pbe$ for $\fa\leq \fb$ (\textit{i.e.}, $a_i\leq b_i$ for all $i$);
	\item $\pae\otimes \cO_Y(D_i)=\mathcal{P}_{\! \fa+1_i}E$ with $1_i =(0,\ldots, 1, \ldots, 0)$ with $1$ in the \supth{i} component;
		\item $\mathcal{P}_{\!\bm{a}+\bm{\epsilon}}E = \pae$ for any vector $\bm{\epsilon}=(\epsilon, \ldots, \epsilon)$ with $0<\epsilon\ll 1$; 
		\item  write $\mathcal{P}_{\!<\bm{a}}E=\cup_{\bm{b}<\bm{a}}\pbe$;   the set of {\em weights} $\bm{a}$ such that $\pae/\mathcal{P}_{\!<\bm{a}}E\not= 0$ is  discrete in $\mathbb{R}^c$.
	\end{thmlist}
\end{dfn}

A weight is normalized if it lies in $(-1,0]^c$. Denote $\mathcal{P}_{\!\bm{0}}E$ by $\diae$, where $\bm{0}=(0, \ldots, 0)$.  Note that the set of weights of $(E,\pae)$ is uniquely determined by the weights lying in $(-1,0]^c$.  

\begin{dfn}
	A {\em parabolic Higgs bundle} on $(Y,D)$ is a filtered 
	bundle $(E,\pae)$ together with $\cO_Y$ linear map
	$$\theta\colon \diae\lra \Omega_Y^1(\log D)\otimes \diae$$
	such that
	$$\theta\wedge \theta=0\quad\text{and}\quad
        \theta(\pae)\subseteq \Omega_Y^1(\log D)\otimes \pae\;\;\text{for }
        \bm{a}\in [-1, 0)^c.$$
\end{dfn}

A natural class of filtered bundles comes from extensions of   systems of Hodge bundles, which will be discussed in  \Cref{sec:prolong}.

 \subsection{Admissible coordinates}

 \begin{dfn}[Admissible coordinate]\label{def:admissible} 
 	Let $(Y,D=\sum_{i=1}^{c}D_i)$ be log pair. Let $p$ be a point of $Y$, and let $\{D_{j}\}_{ j=1,\ldots,\ell}$ 
 	be the components of $D$ containing $p$. An \emph{admissible coordinate} around $p$
 	is a tuple $(\cU;z_1,\ldots,z_n;\varphi)$ (or simply  $(\cU;z_1,\ldots,z_n)$ if no confusion arises) where
 	\begin{itemize}
 		\item $\cU$ is an open subset of $Y$ containing $p$; 
 		\item there is a holomorphic isomorphism   $\varphi\colon \cU\to \Delta^n$ so that  $\varphi(D_j)=(z_j=0)$ for any
 		$j=1,\ldots,\ell$.
 	\end{itemize} 
 	We shall write $\cU^*:=\cU-D$.
 \end{dfn}
 
 Recall that the complete Poincar\'e metric $\omega_P$ on $(\Delta^*)^\ell\times \Delta^{n-\ell}$ is described as 
 \begin{align}\label{eq:Poincare}
 \omega_P=\sum_{j=1}^{\ell}\frac{\sqrt{-1}dz_j\wedge d\bar{z}_j}{|z_j|^2(\log |z_j|^2)^2}+\sum_{k=\ell+1}^{n} \sqrt{-1}\frac{dz_k\wedge d\bar{z}_k  }{(1-|z_k|^2)^2} 
\end{align}
 Note that 
 $
 \omega_P=\hess \varphi
 $ 
 with \begin{align}\label{eq:potential}
 	 \varphi:=- \log \left(\prod_{j=1}^{\ell}\left(-\log |z_j|^2\right)\cdot  \left(\prod_{k=\ell+1}^{n} \left(1-|z_k|\right)^2\right)\right).
 \end{align}
 
 \begin{rem}[Global K\"ahler metric with Poincar\'e growth]\label{rem:Poincare}
 	Let   $(Y,\omega)$ be a  compact K\"ahler  manifold, and let $D=\sum_{i=1}^\ell D_i$ be a simple normal crossing divisor on $Y$.  
        Let $\sigma_i$ be the section $H^0(Y,\cO_Y(D_i))$ defining $D_i$, and pick any smooth metric $h_i$ for the line bundle $\cO_Y(D_i)$.  One can prove that when $\ep>0$ is small enough, the closed $(1,1)$-current
        \begin{align}\label{eq:metric}
 	T:=\omega-\hess \log \left(-\prod_{i=1}^{\ell}\log  |\ep\cdot \sigma_i|_{h_i}^2\right)
 	\end{align} 
 is a  K\"ahler current (\textit{i.e.}, $T\geq \delta  \omega$ for some $\delta>0$), and on any admissible coordinate $(\cU; z_1,\ldots,z_n)$, 
  $
 T|_{\cU-D}$ is mutually bounded with $\omega_P
 $.
 \end{rem}
 
\subsection{Extension of systems of Hodge bundles}\label{sec:prolong}

Let $(Y,D=\sum_{i=1}^{\ell}D_i)$ be log pair.  Let $(E,h)$ be a hermitian bundle on $Y-D$.	For any $\fa=(a_1,\ldots,a_\ell)\in \mathbb{R}^\ell$, we can prolong $E$ over $Y$ by $\pahe$  as follows: 
\begin{align} \label{def:prolong}
\pahe(\cU)=\left\{\sigma\in\Gamma(\cU-D,E|_{\cU-D})\,\mid\, |\sigma|_h\lesssim\frac{1}{\prod_{i=1}^{\ell}|z_i|^{a_i+\ep}}\  \forall \ep>0\right\}, 
\end{align}
where $(\cU; z_1,\ldots,z_n)$ is any admissible coordinate. 
We still use the notation $\diae$ in the case $\bm{a}=(0,\ldots,0)$. In general, $\pahe$ is not coherent. However, by the deep work of Simpson \cite[Theorem 3]{Sim88} and Mochizuki \cite{Moc07}, this is the case for systems of Hodge bundles.
  
\begin{thm}[Simpson, Mochizuki]\label{thm:SM}
If\, $\shb$ is a system of Hodge bundles on $Y-D$, then $(E,\pahe,\theta)$ is a parabolic Higgs bundle on $(Y,D)$. 
\end{thm}

In this case, we  write  $\pae$ for $\pahe$ to lighten the notation, and denote  by
$$
\theta\colon\pae\lra \pae\otimes \Omega_Y^1(\log D)
$$
the prolonged Higgs field by abuse of notation. From \Cref{thm:SM}, one can easily deduce the following. 

\begin{lem}\label{lem:Prolong}
Let $\shb$ be as above. 
\begin{thmlist}
	\item \label{splitting} We have $\pae=\oplus_{p+q=m}\pae^{p,q}$. Here $\pae^{p,q}$ is the extension of\, $(E^{p,q}, h_{p,q})$.
	\item The map $\theta$ restricts to 
          $$
	\theta|_{\pae^{p,q}}\colon  \pae^{p,q}\lra  \pae^{p-1,q+1}\otimes \Omega_{Y}^1(\log D).  
	$$   
	\end{thmlist} 
\end{lem} 

\begin{rem}
	If $\shb$ is  a system of Hodge bundles, $\pae$ coincides with the Deligne extension with real part of the eigenvalue in $[\bm{a},\bm{a}+\bm{1})$.   See the table in \cite[p.~746]{Sim90}.
\end{rem}

\begin{dfn}\label{def:canonical}
Let $(Y,D)$ be a log pair.	Let $\shb$ be a system of Hodge bundles defined over $Y-D$. The extension $(\diae=\oplus_{p+q=m}\diae^{\p,q}, \theta)$ is called the \emph{canonical extension} of $\shb$.
\end{dfn}

\Cref{lem:Prolong} inspires us to introduce the definition of systems of log Hodge bundles.

\begin{dfn}[System of log Hodge bundles]\label{def:log Hodge}
  Let $(Y,D)$ be a log pair.	A \emph{system of log Hodge bundles} of weight $m$ over $(Y,D)$ consists of a pair $(E=\oplus_{p+q=m}E^{p,q},\theta=\oplus_{p+q=m}\theta_{p,q})$, where
  \begin{itemize}
 		\item $E=\oplus_{p+q=m}E^{p,q}$ is a direct sum of holomorphic vector bundles $E^{p,q}$ on $Y$;
 		\item $\theta$ is a direct sum of  $$
 		\theta_{p,q}\colon E^{p,q}\lra E^{p-1,q+1}\otimes \Omega^1_Y(\log D)
 		$$
 		with $\theta\wedge\theta=0$. 
 	\end{itemize}  
 \end{dfn} 

\section{Construction of a special system of log Hodge bundles}
In this section, we first study the refined positivity for the Griffiths line bundle associated to a system of Hodge bundles. This positivity is well known when the corresponding $\bC$-PVHS has unipotent monodromies near the boundary. We then construct a special system of log Hodge bundles (see \cref{prop:new}) over the log pair $(Y,D)$ in   \cref{thmx:uniform}. Such a system of Hodge bundles will be used to construct a negatively curved Finsler metric in \Cref{sec:construction}.

\subsection{Refined positivity for Griffiths line bundles} 
For a system of Hodge bundles $\shb$ over a complex manifold $U$, in \cite{Gri70}  Griffiths constructed a line bundle $\kL$ on $U$, which can be endowed with a natural metric with semipositive curvature. Precisely, one has  
\begin{align}\label{eq:Griffiths}
\kL:= \left(\det E^{m,0}\right)^{\otimes m}\otimes\left(\det E^{m-1,1}\right)^{\otimes(m-1)}\otimes\cdots\otimes  \det E^{1,m-1}.  
\end{align}
Here $\theta_{p,q}^*$ is the adjoint of $\theta_{p,q}$ with respect to $h_{p,q}$. 
 The Hodge metric $h$ then induces a metric $h_{\kL}$ on $\kL$  whose curvature is  
\begin{align}\label{eq:Gc}
\sn \Theta_{h_{\kL}}(\kL)=-{\tr}\left(\sum_{q=0}^{m-1}\theta_{m-q,q}^*\wedge \theta_{m-q,q}\right)\geq 0.
\end{align}
One can  see that $\sn \Theta_{h_{\kL}}(\kL)>0$ at the point $y$ where $\theta\colon T_{Y,y}\to \End(E_y)$ is injective. Note that $\theta$ is the differential of the period map  (see, \textit{e.g.}, \cite[p. 429]{KKM11} for the proof). This means that $\sn \Theta_{h_{\kL}}(\kL)$ is strictly positive at the point where the period map is immersive.
 
  Now assume  $U=Y-D$, where $(Y,D)$ is a compact K\"ahler log pair. Let $T$ be the K\"ahler current on $Y$ defined in \cref{rem:Poincare}. Then $\omega_U:=T|_{U}$ is a complete K\"ahler metric with Poincar\'e type near $D$. We recall the following theorem by Simpson \cite[Lemma 10.1]{Sim88} and Mochizuki \cite{Moc07}.

  \begin{thm}[Simpson, Mochizuki]\label{thm:norm}
  	Let $\shb$ be a system of Hodge bundles on $U=Y-D$. Then 
  \begin{align}
  |\theta|_{h,\omega_U}\leq C
  \end{align}
  for some constant $C>0$. 
  \end{thm}
  
\begin{lem}\label{lem:bounded}
  In the notation above, $ \sn \Theta_{h_{\kL}}(\kL)$ is less singular than $\omega_{U}$, which we denote by
$$\sn \Theta_{h_{\kL}}(\kL)\lesssim \omega_{U}.$$ 
\end{lem}

\begin{proof}
By \cref{thm:norm}, one has
 $
|\theta_{p,q}|_{h,\omega_{U}}\leq C.
$ 
Then  $
|\theta_{p,q}^*|_{h,\omega_{U}}\leq C.
$ Hence 
$$
\left|\theta_{p,q}^*\wedge \theta_{p,q}\right|_{h,\omega_{U}}\leq C^2.
$$
It follows from \eqref{eq:Gc} that
$$
\left|\sn \Theta_{h_\kL}(\kL)\right|\leq C'
$$ 
for some constant $C'>0$. The lemma follows directly from the above inequality.
\end{proof}

By \cref{lem:bounded}, the mass of $\sn \Theta_{h_{\kL}}(\kL)$ is bounded near $D$, and one can thus apply   the Skoda extension theorem (see \cite[Theorem 2.3]{Dem97}) so that the trivial extension of $\sn \Theta_{h_{\kL}}(\kL)$ over $Y$ is a positive closed $(1,1)$-current, which is denoted by $S$.  The current $S$ is therefore less singular than the current $T$ defined in \cref{rem:Poincare}, which we denote by $S\lesssim T$.

Let us consider the extension $\pf\kL$ of $(\kL,h_{\kL})$ defined in \eqref{def:prolong}, where $\bm{1}=(1,\ldots,1)$. Then $h_{\kL}$ can be seen as the \emph{singular hermitian metric} for $\pf\kL$; this can be seen explicitly from the proof of the next lemma.  

\begin{lem}\label{lem33}
The curvature $\sn \Theta_{h_{\kL}}(\pf\kL)$ is a closed positive $(1,1)$-current. In particular, $\pf\kL$ is  a pseudo effective line bundle on $Y$.
\end{lem}

\begin{proof}
	Pick any $p\in Y$. We take an admissible  coordinate $(W;z_1,\ldots,z_n)$ around $p$ as in \cref{def:admissible}. Since $S$ is a closed positive current on $Y$, over $W$ there is a plurisubharmonic function $\psi$ so that $S=\hess \psi$.	Note that $S\lesssim T$. One thus has
 $
\varphi\lesssim \psi
$,  
where $\varphi$ is defined in \eqref{eq:potential}.  For the new metric
 $
	\tilde{h}_\kL:=h_\kL\cdot e^{\psi} 
	$ of $\kL$, 
	one has $\Theta_{\tilde{h}_\kL}(\kL)= 0$ over $W-D\simeq (\Delta^*)^\ell\times \Delta^{n-\ell}$.  
	
	 Let $\nabla$ be the corresponding Chern connection of $(\kL,\tilde{h}_\kL)$, which is flat by the relation $\Theta_{\tilde{h}_\kL}(\kL)= 0$.  It corresponds to a unitary representation $\rho\colon \bZ^\ell\simeq \pi_1(W-D)\to U(1)$.    Let $\gamma_i$ be a clockwise loop around the origin in the $\supth{i}$ factor $   (\Delta^*)^\ell\times \Delta^{n-\ell}\simeq W-D$. Let $T_i=\rho([\gamma_i])\in U(1)$ be the monodromy  corresponding to the loop $\gamma_i$.

Consider the universal covering map 
\begin{align*}
\pi\colon	\mathbb{H}^\ell\times \Delta^{n-\ell}&\lra(\Delta^*)^\ell\times \Delta^{n-\ell}\\
	(t_1,\ldots,t_\ell,z_{\ell+1},\ldots,z_n)&\longmapsto \left(e^{2\pi\sn  t_1},\ldots,e^{2\pi\sn  t_\ell},z_{\ell+1},\ldots,z_n\right),
\end{align*}
where $\mathbb{H}=\{t\in \bC\mid {\Im}(z)>0\}$. 	Choose a flat section $\Phi$ of the flat line bundle $\pi^*(\kL,\nabla)$. Since $(\kL, \tilde{h}_{\kL})$ is unitary flat,   $|\Phi|_{\tilde{h}_{\kL}}$ is constant, and we may assume that $|\Phi|_{\tilde{h}_{\kL}}\equiv 1$. Recall that $T_i=\rho([\gamma_i])\in U(1)$ is the monodromy  corresponding to the loop $\gamma_i$; one has
\begin{align}\label{eq:monodromy}
	T_i\cdot \Phi(t_1,\ldots,t_\ell,z_{\ell+1},\ldots,z_n)=\Phi(t_1,\ldots,t_{i}+1,\ldots,t_\ell,z_{\ell+1},\ldots,z_n). 
\end{align}
Write $T_i=e^{2\pi\sn  b_i}$ for some $0<b_i\leq1$. 
	Define  a new section of $\pi^*\kL$ by
	$$
\Psi(t_1,\ldots,t_\ell,z_{\ell+1},\ldots,z_n):=\Phi(t_1,\ldots,t_\ell,z_{\ell+1},\ldots,z_n)e^{-2\pi\sn \sum_{i=1}^{\ell} b_it_i}.
	$$
By \eqref{eq:monodromy}, one has
$$
\Psi(t_1,\ldots,t_\ell,z_{\ell+1},\ldots,z_n):=\Psi(t_1,\ldots,t_{i}+1,\ldots,t_\ell,z_{\ell+1},\ldots,z_n)
$$  
for any $i=1,\ldots,\ell$. 
It thus  descends to  a section $\sigma(z)$ of $\kL|_{W-D}$; \textit{i.e.}, 
$$
\sigma(\pi(t_1,\ldots,t_\ell,z_{\ell+1},\ldots,z_n))=\Psi(t_1,\ldots,t_\ell,z_{\ell+1},\ldots,z_n).
$$
Note that $\nabla(\Phi)=0$; one has
 $$
 \nabla(\Psi)=\Phi\cdot e^{-2\pi\sn \sum_{i=1}^{\ell} b_it_i} \left(-2\pi\sn \sum_{i=1}^{\ell} b_i\cdot dt_i\right)=\Psi\cdot \left(-2\pi\sn \sum_{i=1}^{\ell} b_i\cdot dt_i\right).
 $$
 Hence 
	$$
	\nabla (\sigma(z))=-\sum_{i=1}^{\ell}b_i d\log z_i\cdot \sigma(z).
	$$
Therefore, $\sigma(z)$ is a holomorphic section trivializing $\kL|_{W-D}$. 
Note that 
	$$
	|\Psi|_{\tilde{h}_{\kL}}=\left|\Phi\cdot e^{-2\pi \sn \sum_{i=1}^{\ell} b_it_i}\right|= \left|e^{-2\pi \sn \sum_{i=1}^{\ell} b_it_i}\right|,
	$$
	where the second equality follows from the fact that $	|\Phi|_{\tilde{h}_{\kL}}\equiv 1$.  It follows that
  $|\sigma(z)|_{\tilde{h}_{\kL}}=\prod_{i=1}^{\ell}|z_i|^{-b_i}$, and thus
\begin{align}\label{norm}
	|\sigma(z)|_{h_{\kL}} =\prod_{i=1}^{\ell}|z_i|^{-2b_i}\cdot e^{-\psi}
\end{align}
by the relation $
\tilde{h}_\kL:=h_\kL\cdot e^{\psi} 
$. 
Since $\varphi\lesssim \psi$,  one has
$$1\lesssim e^{-\psi}\lesssim e^{-\varphi}\lesssim \left(\prod_{j=1}^{\ell}\left(-\log \left|z_j\right|^2\right)\right)^N$$
for some $N>0$, where the last inequality follows from \eqref{eq:potential}.  Therefore,
\begin{align}\label{eq:norm}
	\prod_{i=1}^{\ell}\frac{1}{|z_i|^{b_i-\ep}}\lesssim|\sigma(z)|_{h_{\kL}}\lesssim \prod_{i=1}^{\ell}\frac{1}{|z_i|^{b_i+\ep}}
\end{align}
	for any $\ep>0$. Since $0<b_i\leq 1$, one has $\sigma\in \pf\kL|_{W}$ by \eqref{def:prolong}. Let us show that $\sigma$ is a generator of $\pf\kL|_{W}$.
	
	For any section $s\in \pf\kL(W)$, there is a holomorphic function  $f\in \cO(W-D)$ so that $s=f\cdot \sigma$. By \eqref{def:prolong} again, 
	$$
|f|\cdot |\sigma|_{h_\kL}=|s|_{h_{\kL}}  \lesssim\frac{1}{\prod_{i=1}^{\ell}|z_i|^{1+\ep}}  	$$
	for all $\ep>0$. By \eqref{eq:norm}, one has
	$$
|f|  \lesssim\frac{1}{\prod_{i=1}^{\ell}|z_i|^{1-b_i+\ep}}  	$$
		for all $\ep>0$. Pick $\ep\ll 1$ with $1-b_i+\ep<1$ for all $i$. The above inequality shows  that $f$ extends to a holomorphic function over $W$. Hence $\sigma$ is a generator of $\pf\kL|_{W}$. 
		
	 By \eqref{norm}, one has
\begin{align}\label{eq:Siu}
	\sn \Theta_{h_\kL}(\pf\kL)=\hess \log |\sigma|_{h_{\kL}}=S+\sum_{i=1}^{\ell}b_i[D_i], 
	\end{align}
	where $[D_i]$ is the current of integration associated to $D_i$. This finishes the proof of the theorem.
	\end{proof}  

The following lemma is  a consequence of the above proof.

\begin{lem}\label{lem:computation}
  For any $N\in \bZ_{> 0}$, let $\pf(\kL^{\otimes N})$ be the extension of\, $(\kL^{\otimes N}, h_\kL^{\otimes N})$ defined in \eqref{def:prolong}. Then
  $$
  \pf\left(\kL^{\otimes N}\right)=(\pf\kL)^{\otimes N}\otimes \cO\left(-\sum_{i=1}^{\ell}(\,\lceil Nb_i\rceil-1) D_i\right). $$
\end{lem}

\begin{proof}
We use the same notation as that in the proof of Lemma~\ref{lem33}.	Consider the section $\sigma^N$, which is a generator of $(\pf\kL)^{\otimes N}|_{W}$.  	For any section $s\in \pf(\kL^{\otimes N})(W)$, there is a holomorphic function  $f\in \cO(W^*)$ so that $s=f\cdot \sigma^N$, where $W^*:=W-D$. By \eqref{def:prolong} again, one has
$$
|f|\cdot |\sigma^N|_{h^{\otimes N}_\kL}=|s|_{h^{\otimes N}_{\kL}}  \lesssim\frac{1}{\prod_{i=1}^{\ell}|z_i|^{1+\ep}}  	$$
for all $\ep>0$. By \eqref{eq:norm}, one has
$$
|f|  \lesssim\frac{1}{\prod_{i=1}^{\ell}|z_i|^{1-Nb_i+\ep}}  	$$
for all $\ep>0$. This shows  that $f\in \cO_Y(-\sum_{i=1}^{\ell}(\lceil Nb_i\rceil-1) D_i)$.

On the other hand, if $g\in \cO_Y(-\sum_{i=1}^{\ell}(\lceil Nb_i\rceil-1) D_i)$, then by \eqref{eq:norm}, one has
$$
|g\cdot \sigma^N|  \lesssim\frac{1}{\prod_{i=1}^{\ell}|z_i|^{1-\lceil Nb_i\rceil+Nb_i+\ep}}  \lesssim\frac{1}{\prod_{i=1}^{\ell}|z_i|^{1+\ep}}
$$
for any $\ep>0$. This yields the lemma. 
	\end{proof}

In summary, we have the following positivity result for Griffiths line bundles.

\begin{proposition}\label{prop:bigness}
Let $(Y,D)$ be a compact  K\"ahler log pair. Let $\shb$ be a system of Hodge bundles  over   $Y-D$. 	Assume that its period map is immersive at one point. Then
 $
	\dia(\kL^{\otimes N})\otimes \cO_Y(-D)
	$  is a big line bundle on $Y$ for $N\gg 1$. In particular, $Y$ is projective.
\end{proposition}

\begin{proof}
Recall that the closed positive current $S$ is the trivial extension of the semipositive $(1,1)$-form $\Theta_{h_\kL}(\kL)$ over $Y$.  By \eqref{eq:Siu}, one has
$$
\{c_1(\pf\kL)\}=\{S\}+\sum_{i=1}^{\ell}b_i\{D_i\}.
$$
\cref{lem:computation} then yields 
$$
\left\{c_1\left(\pf\left(\kL^{\otimes N}\right)\right)\right\}=N\{S\}+\sum_{i=1}^{\ell}(Nb_i-\lceil Nb_i\rceil+1)\{D_i\}.
$$
 Note that 
$$\pf\left(\kL^{\otimes N}\right)=\dia\left(\kL^{\otimes N}\right)\otimes \cO_Y(D).$$
Therefore,
$$c_1\left(\dia\left(\kL^{\otimes N}\right)\otimes \cO_Y(-D)\right)=N\{S\}+\sum_{i=1}^{\ell}(-1+Nb_i-\lceil Nb_i\rceil)\{D_i\}.$$

By the   discussion at the beginning of this subsection, the semipositive $(1,1)$-form $\Theta_{h_\kL}(\kL)$  is strictly positive at the point where the period map is immersive. By Boucksom's criterion \cite{Bou02},   the cohomology class $\{S\}$ is a big $(1,1)$-class. Therefore,  
 $
N\{S\}-2D
$ 
is big for $N\gg 1$. Note that
$$
1+Nb_i-\lceil Nb_i\rceil\geq 0.
$$
Since the sum of a big class with an effective class is still big, we conclude that  $c_1(\dia(\kL^{\otimes N})\otimes \cO_Y(-D))$ is big. This proves the lemma.
	\end{proof} 

\subsection{Special system of Hodge bundles} 
  Let $(Y,D)$ be a compact K\"ahler log pair.   Let $(F=\oplus_{p+q=m}F^{p,q},\eta,h_F)$ be a system of  Hodge bundle  over $U:=Y-D$ whose period map is  immersive at one point.    Let us write $r_p:=\rank\, F^{p,q}$. 
Recall that the Griffiths line bundle for $(F=\oplus_{p+q=m}F^{p,q},\eta,h_F)$  is 
$$
\kL:= \left(\det F^{m,0}\right)^{\otimes m}\otimes\left(\det F^{m-1,1}\right)^{\otimes(m-1)}\otimes\cdots\otimes  \det F^{1,m-1}.
$$
By \cref{prop:bigness}, $\dia(\kL^{\otimes N})\otimes \cO_Y(-D)$ is a big line bundle for some $N\gg 1$. Let us write \[r:=N\big(mr_m+(m-1)r_{m-1}+\cdots+r_1\big).\]

 We define a new system of  Hodge bundle $(E=\oplus_{P+Q=rm}E^{P,Q},\theta,h)$ on $U=Y-D$ by setting $(E,\theta,h):=({F},{\eta},h_F)^{\otimes r}$. Precisely,  $E:=F^{\otimes r}$, 
and 
$$
\theta:=\eta\otimes \underbrace{\mathds{1}\otimes\cdots\otimes\mathds{1}}_{(r-1)-\text{tuple}}+\mathds{1}\otimes\eta\otimes \underbrace{\mathds{1}\otimes\cdots\otimes\mathds{1}}_{(r-2)-\text{tuple}}+\cdots+ \underbrace{\mathds{1}\otimes\cdots\otimes\mathds{1}}_{(r-1)-\text{tuple}}\otimes \eta.
$$
Define
\begin{align}\label{eq:decomposition}
E^{P,Q}:=\oplus_{p_1+\cdots+p_r=P;q_1+\cdots+q_r=Q}F^{p_1,q_1}\otimes\cdots\otimes  F^{p_r,q_r}.
\end{align} 
Then we have 
$$
\theta\colon E^{P,Q}\lra E^{P-1,Q+1}\otimes \Omega^1_{U},
$$
and one can easily check that $h=h_F^{\otimes r}$ is the Hodge metric for $(E=\oplus_{P+Q=rm}E^{P,Q},\theta)$.

Note that $\det F^{p,q}=\wedge^{r_p}F^{p,q}\subset (F^{p,q})^{\otimes r_p}\subset F^{\otimes r_p}$. Hence   
\begin{align*}
	\kL^{\otimes N} &= \left(\det F^{m,0}\right)^{\otimes Nm}\otimes\left(\det F^{m-1,1}\right)^{\otimes(N(m-1))}\otimes\cdots\otimes  \left(\det F^{1,m-1}\right)^{\otimes N}\\
	 &\subset \left(F^{m,0}\right)^{\otimes Nmr_m}\otimes\cdots\otimes \left(F^{1,m-1}\right)^{\otimes Nr_1}\subset E^{P_0,Q_0}, 
\end{align*}
where $P_0=N(r_mm^2+r_{m-1}(m-1)^2+\cdots+r_1)$ and $Q_0=rm-P_0$. In other words, $\kL^{\otimes N}$ is a subbundle of $E^{P_0,Q_0}$. Moreover, their hermitian metrics are compatible in the following sense:
 $h_{\kL}^{\otimes N}=h|_{\kL}$.  
By the very definition of the extension \eqref{def:prolong}, one has 
$$
\dia \left(\kL^{\otimes N}\right)\subset \diae^{P_0,Q_0}.
$$

In summary, we construct a \emph{special system of log Hodge bundles} on  $(Y,D)$ as follows (we change the notation for  brevity's sake). 

\begin{thm}\label{prop:new}
Let $(Y,D)$ be a compact K\"ahler log pair.  Let $(F=\oplus_{p+q=m}F^{p,q},\eta,h_F)$ be a system of  Hodge bundles  over $Y-D$ whose period map is  immersive at one point.  Then there is a system of log Hodge bundles $(E=\oplus_{p+q=\ell}E^{p,q},\theta=\oplus_{p+q=\ell}\theta_{p,q})$ on   $(Y,D)$   satisfying the following properties: 
\begin{thmlist}
	\item \label{cond:monodromy} The pair  $(E,\theta)$ is the canonical extension $($in the sense of \cref{def:canonical}\,$)$   of some system of   Hodge bundles $(\tilde{E},\tilde{\theta},h_{\hod})$ defined over $Y-D$. 
	\item \label{cond:big} There is a big   line bundle $L$ over $Y$ such that $L\subset E^{p_0,q_0}$ for some $p_0+q_0=\ell$, and $L\otimes \cO_Y(-D)$ is still big.
 	\item  \label{augmented base} If the period map  moreover has zero-dimensional fibers, then the augmented base locus satisfies $\mathbf{B}_+(L)\subset D$.
	\end{thmlist}
	\end{thm}

\begin{rem}\label{rem:analogue} The interested readers can compare the Higgs bundle in \cref{prop:new}  with  the Viehweg--Zuo Higgs bundle in \cite{VZ02,VZ03} (see also \cite{PTW18}). Loosely speaking,  a \emph{Viehweg--Zuo Higgs bundle} for a log pair $(Y,D)$ is a Higgs bundle $(E=\oplus_{p+q=m}E^{p,q},\theta)$ over $(Y, D+S)$   induced by some (geometric) $\bZ$-PVHS defined over a Zariski open subset of $Y-(D\cup S)$, where $S$ is another divisor on $Y$ so that $D+S$ is  simple normal crossing. The extra data is that there is a sub-Higgs sheaf $(F=\oplus_{p+q=m} F^{p,q},\eta)\subset (E,\theta)$ such that the \emph{first stage} $F^{n,0}$ is a big line bundle, and that we have 
	$$\eta\colon F^{p,q}\lra F^{p-1,q+1}\otimes \Omega_{Y}^1(\log D).$$
	As we explained in \Cref{sec:describe}, the positivity $F^{n,0}$   comes in a sophisticated way from   Kawamata's big line bundle $\det f_*(mK_{X/Y})$, where $f\colon X\to Y$ is some algebraic fiber space between projective manifolds.  For our Higgs bundle $(E=\oplus_{p+q=m}E^{p,q},\theta)$ over the log pair $(Y,D)$ in \cref{prop:new}, the global positivity is the Griffiths line bundle which is  contained in some \emph{intermediate stage} $E^{p_0,q_0}$ of $(E=\oplus_{p+q=m}E^{p,q},\theta)$.
\end{rem}

\subsection{Iterating Higgs fields}\label{sec:iterate}
 Let $(E=\oplus_{p+q=\ell}E^{p,q},\theta)$ be the system of log Hodge bundles on a compact K\"ahler log pair $(Y,D)$ satisfying the  two conditions in \cref{prop:new}. We apply   ideas  by Viehweg--Zuo \cite{VZ02,VZ03} to iterate Higgs fields.

Since we have $\theta\colon E^{p,q}\to E^{p-1,q+1}\otimes \Omega_Y^1(\log D)$, one can iterate $\theta$   $k$ times to obtain
$$
 E^{p_0,q_0}\lra E^{p_0-1,q_0+1}\otimes \Omega_Y^1(\log D)\lra \cdots \lra E^{p_0-k,q_0+k}\otimes  \otimes^k\Omega_Y^1(\log D).
$$
Since $\theta\wedge\theta=0$, the above morphism factors through
\begin{align}\label{iterate}
 E^{p_0,q_0}  \lra E^{p_0-k,q_0+k}\otimes {\Sym}^k\Omega_Y^1(\log D).
\end{align}
Since $L$ is a subsheaf of $E^{p_0,q_0}$, it induces
$$
 L\lra E^{p_0-k,q_0+k}\otimes {\Sym}^k\Omega_Y^1(\log D),
$$
which is equivalent to a morphism
\begin{align}\label{iterated Kodaira2} 
\tau_k\colon  {\Sym}^k T_Y(-\log D)\lra L^{-1}\otimes E^{p_0-k,q_0+k}.
\end{align}
The readers might be worried that all $\tau_k$ might be trivial, so that the above construction will be meaningless. In the next subsection, we will show that this  cannot happen.

\subsection{An infinitesimal Torelli-type theorem}   \label{sec:Torelli}

We begin with the following technical lemma.

  \begin{proposition}\label{singular metric}
  Let  $(E=\oplus_{p+q=\ell}E^{p,q},\theta)$ be a system of log Hodge bundles on  a compact K\"ahler log pair $(Y,D)$ satisfying the  two conditions in \cref{prop:new}. Then there is a singular hermitian metric $h_L$ with analytic singularities  for $L$ such that
  \begin{thmlist} 
		\item\label{estimate}    the curvature current  satisfies
		\[
		\sqrt{-1}\Theta_{h_{L}}(L) \geq T,
		\]
		where $T$ is the K\"ahler current on $Y$ defined in \cref{rem:Poincare}; 
		\item \label{bounded} the singular hermitian metric \(h:=h_{L}^{-1}\otimes h_{\hod} \) on \(  L^{-1}\otimes E\) is locally bounded  on \(Y\) and smooth outside \(D\cup \mathbf{B}_+(L-D)\), where $h_{\hod}$ is the Hodge metric for the system of Hodge bundles $(E=\oplus_{p+q=\ell}E^{p,q},\theta)|_{U}$.  Moreover, \(h\cdot \prod_{i=1}^{\ell}|\sigma_i|_{h_i}^{-\ep}\) vanishes  on \(D\cup \mathbf{B}_+(L-D)\) for $\ep>0$ small enough.  	Here $\sigma_i$ is  the canonical section in $H^0(Y,\cO_Y(D_i))$ defining $D_i$, and   $h_i$ is a  smooth metric for the line bundle $\cO_Y(D_i)$.  
	\end{thmlist}  
\end{proposition} 

\begin{proof}
	By \cref{cond:big},  the line bundle \(L\otimes \cO_Y(-D)\) is big, and thus by \cite[Theorem 3.17]{Bou04}, we can put a  singular hermitian metric $g_0$ on it with analytic singularities for $L\otimes \cO_Y(-D)$ such that \(g_0\) is smooth on $Y\setminus \mathbf{B}_+(L\otimes \cO_Y(-D))$, where $\mathbf{B}_+(L\otimes \cO_Y(-D))$ is the augmented base locus of $L\otimes \cO_Y(-D)$,  and the curvature current satisfies
	$
	\sqrt{-1}\Theta_{g_0}(L-D)\geqslant \omega
	$
	for some smooth K\"ahler form \(\omega \) on $Y$.   Take $g:=g_0(-\prod_{i=1}^{\ell}\log  |\ep\cdot \sigma_i|_{\cdot h_i}^2)$. Then 
	$$
	\sn \Theta_{g}(L-D)\geqslant T:=\omega-\hess \log \left(-\prod_{i=1}^{\ell}\log  |\ep\cdot \sigma_i|_{h_i}^2\right).
	$$ 
	Note that  $T$ is a K\"ahler current when $0<\ep\ll 1$.
	
	Let $h_D$ be the canonical singular hermitian metric for $D$ so that $\sn \Theta_{h_D}(\cO_Y(D))=[D]$.  We  define a singular hermitian metric on $L$ as follows:
	\[
	h_L:=g\cdot h_D.
	\] 
	Then 	\[
	\sqrt{-1}\Theta_{h_{L}}(L)=	 
	\sqrt{-1}\Theta_{g}(L\otimes \cO_X(-D))+[D]  \geq T.
	\]
	The first condition is verified. 
	
Note that $g^{-1}$ vanishes on $\mathbf{B}_+(L\otimes \cO_Y(-D))$, and $h_D^{-1}$ vanishes on $D$.	Since $h_{\hod}$ is smooth over $Y-D$, we have $\mathbf{B}_+(L)\subset \mathbf{B}_+(L\otimes \cO_Y(-D))$, so $h:=h_{\hod}\cdot h_L^{-1}$  vanishes on $\mathbf{B}_+(L)-D$. For any point $y\in D$, we pick an admissible coordinate $(W;z_1,\ldots,z_n)$ and a frame $(e_1,\ldots,e_r)$ for $E|_{W}$. Since $(E,\theta)$ is the canonical extension of a system of Hodge bundles $(\tilde{E},\tilde{\theta},h_{\hod})$, by \eqref{def:prolong} one has 
$$
|e_i|_h  \lesssim\frac{1}{\prod_{i=1}^{\ell}|z_i|^{\ep}}
$$
for all $\ep>0$.    Pick a section $e\in L(W)$ which trivializes $L|_{W}$. By the definition of $h_L$, one has
$$
|e|_{h_L}^2\gtrsim  \frac{1}{\prod_{i=1}^{\ell}|z_i|}. 
$$
Hence for the frame $(e_1\otimes e^{-1},\ldots,e_r\otimes e^{-1})$ trivializing $E\otimes L^{-1}|_W$, one has
$$
|e_i\otimes e^{-1}|_h  \lesssim {\prod_{i=1}^{\ell}|z_i|^{1-\ep}}
$$
for any $\ep>0$. This shows that $h\cdot \prod_{i=1}^{\ell}|\sigma_i|_{h_i}^{-\ep}$ vanishes on $D$ when $\ep>0$ is small enough. The proposition is proved.
	\end{proof}

\begin{thm}[Infinitesimal Torelli-type property]\label{Deng}
	The morphism $\tau_1\colon T_Y(-\log D)\to L^{-1}\otimes E^{p_0-1,q_0+1}$ defined in \eqref{iterated Kodaira2} is   generically injective.
\end{thm}

The proof is almost the same at that of  \cite[Theorem D]{Den18}. We provide it here for the sake of completeness.

\begin{proof}[Proof]
  The inclusion $   L\subset E^{p_0,q_0}$   induces a global section  $s\in H^0(Y, L^{-1}\otimes E^{p_0,q_0})$ by \Cref{cond:big}; this section is \emph{generically} non-vanishing over $U=Y- D $. Set
	\begin{align}\label{set}
	U_1:=\{y\in Y-(D\cup \mathbf{B}_+(L-D)) \mid s(y)\neq 0 \}, 
	\end{align}
	which is a non-empty Zariski open subset of $U$.  Since the Hodge metric $h_{\hod}$ is a direct sum of metrics $h_p$ on $E^{p,q}$, the metric $h$ for $L^{-1}\otimes E$ is a direct sum of metrics    $h_L^{-1}\cdot h_{p}$ on $L^{-1}\otimes E^{p,q}$, which is smooth over $U_0:=Y-(D\cup \mathbf{B}_+(L-D))$. Let   $D'$ be the $(1,0)$-part of its Chern connection over $U_1$ and $\Theta$ to be its curvature form. Then   over $U_0$, we   have
	\begin{align}\nonumber
	\Theta &=- \Theta_{L,h_L}\otimes \mathds{1}+\mathds{1}\otimes \Theta_{h_{p_0}}\left(E^{p_0,q_0}\right)\\\nonumber
	&=- \Theta_{L,h_L}\otimes \mathds{1}-\mathds{1}\otimes   \left({\theta}_{p_0,q_0}^*\wedge  {\theta}_{p_0,q_0}\right)-\mathds{1}\otimes   \left({\theta}_{p_0+1,q_0-1} \wedge  {\theta}_{p_0+1,q_0-1}^*\right)\\\label{eq:Hodge bundle2}
	&=- \Theta_{L,h_L}\otimes \mathds{1}-   \tilde{\theta}_{p_0,q_0}^*\wedge  \tilde{\theta}_{p_0,q_0} - \tilde{\theta}_{p_0+1,q_0-1} \wedge  \tilde{\theta}_{p_0+1,q_0-1}^*,  
	\end{align}
	where we set
        $$
        \theta_{p,q}=\theta|_{E^{p,q}}\colon E^{p,q}\lra E^{p-1,q+1}\otimes \Omega_{Y}^1(\log D)
        $$
	and
        $$
        \tilde{\theta}_{p,q}=\mathds{1}\otimes \theta_{p,q} \colon L^{-1}\otimes E^{p,q}\lra L^{-1}\otimes  E^{p-1,q+1}\otimes \Omega_{Y}^1(\log D)
	$$
        and define $\tilde{\theta}_{p,q}^*$ to be the adjoint of $\tilde{\theta}_{p,q}$ with respect to the metric $h_{L}^{-1}\cdot h$. Hence over $U_1$, one has
	\begin{align} \nonumber
	-\hess\log |s|_{h}^2&=  \frac{\left\{ \sqrt{-1}\Theta(s),s\right\}_{h}}{| s|^2_{h}}+\frac{\sqrt{-1}\{D's,s \}_{h}\wedge \{s,D's \}_{h}}{|s|_{h}^4}-\frac{\sqrt{-1}\{D's,D's \}_{h}}{|s|_{h}^2} \\\label{eq:crucial}
	& \leqslant   \frac{\left\{ \sqrt{-1}\Theta(s),s\right\}_{h}}{| s|^2_{h}}
	\end{align}	
	thanks to Cauchy--Schwarz inequality 
	$$\sqrt{-1}|s|^2_{h}\cdot \{D's,D's \}_{h}\geqslant \sqrt{-1}\{D's,s\}_{h}\wedge \{s,D's \}_{h}.$$
	Substituting \eqref{eq:Hodge bundle2} into \eqref{eq:crucial}, over $U_1$, one has 
	\begin{align} \nonumber
	\sqrt{-1}\Theta_{L,h_L}-\hess\log |s|_{h}^2 &\leqslant  -\frac{\left\{ \sqrt{-1}\tilde{\theta}_{p_0,q_0}^*\wedge \tilde{\theta}_{p_0,q_0}(s),s\right\}_{h}}{| s|^2_{h}}
        -\frac{\left\{ \sqrt{-1}\tilde{\theta}_{p_0+1,q_0-1}\wedge \tilde{\theta}^*_{p_0+1,q_0-1}(s),s\right\}_{h}}{| s|^2_{h}}\\\nonumber
	&= \frac{\sqrt{-1}\left\{ \tilde{\theta}_{p_0,q_0}(s),\tilde{\theta}_{p_0,q_0}(s)\right\}_{h}}{| s|^2_{h}} + \frac{\left\{   \tilde{\theta}^*_{p_0+1,q_0-1}(s), \tilde{\theta}^*_{p_0+1,q_0-1}(s)\right\}_{h}}{| s|^2_{h}}\\	\label{eq:final}
	&\leq \frac{\sqrt{-1}\left\{ \tilde{\theta}_{p_0,q_0}(s),\tilde{\theta}_{p_0,q_0}(s)\right\}_{h}}{| s|^2_{h}}, 
	\end{align}
	where $\tilde{\theta}_{p_0,q_0}(s)\in H^0\big(Y,L^{-1}\otimes  E^{p_0-1,q_0+1}\otimes \Omega_{Y}^1(\log D)\big)$. 
	By \Cref{bounded}, one has $|s|_h^2(y)=0$ for any $y\in D\cup \mathbf{B}_+(L-D)$. 
	Therefore,   there exists a  $y_0\in U_0$ so that $|s|^2_{h}(y_0)\geqslant |s|^2_{h}(y)$ for any $y \in U_0$. Hence $|s|^2_{h}(y_0)>0$, and by \eqref{set}, $y_0\in U_1$. Since $|s|^2_{h}$ is smooth over $U_0$,  
	$
\hess\log |s|_{h}^2 
	$ 
	is seminegative at $y_0$ by the maximal principle. By \Cref{estimate}, $	\sqrt{-1}\Theta_{L,h_L}$ is strictly positive at $y_0$. By \eqref{eq:final} and the relation $|s|_h^2(y_0)>0$, we conclude  that $\sqrt{-1}\big\{  \tilde{\theta}_{p_0,q_0}(s),\tilde{\theta}_{p_0,q_0}(s)\big\}_{h}$ is strictly positive at $y_0$. In particular, for any non-zero $\xi \in T_{Y,y_0}$, one has $\tilde{\theta}_{p_0,q_0}(s)(\xi)\neq 0$. For $k=1$, we write $\tau_k$ in \eqref{iterated Kodaira2} as
	$$
	\tau_1\colon T_Y (-\log D  )\lra L^{-1}\otimes E^{p_0-1,q_0+1}.
	$$
Then over $U$, it is defined by $\tau_1(\xi):=\tilde{\theta}_{p_0,q_0}(s)(\xi)$ and is thus \emph{injective at $y_0\in U_1$}. Hence $\tau_1$ is \emph{generically injective}. The theorem is thus proved.
\end{proof}

 \section{Construction of a negatively curved Finsler metric}  \label{sec:construction}
The aim of this technical section is to prove \cref{thmx:uniform} based on \cref{prop:new}. We first give the definition of a Finsler metric. 

\begin{dfn}[Finsler metric]\label{def:Finsler}
	Let \(E\) be a  holomorphic vector bundle on a complex manifold $X$. A \emph{Finsler metric} on \(E\) is a real non-negative  \emph{continuous}  function \(h\colon E\to  \mathclose[ 0,+\infty\mathopen[ \) such
	that 
	\[h(av) = |a|h(v)\]
	for any \(a\in \bC \) and \(v\in  E\).  	The metric $h$ is \emph{positive  definite} on a subset $U\subset X$ if  \(h(v)>0\) for any
	non-zero \(v\in E_{x} \) and any $x\in U$.
\end{dfn}

We  mention that our definition is a bit different from that in \cite[Section 2.3]{Kob98}, which requires \emph{convexity}, and the Finsler metric therein    can be upper-semicontinuous. 

Let $(E=\oplus_{p+q=\ell}E^{p,q},\theta)$ be a system of log Hodge   bundles on a compact K\"ahler log pair $(Y,D)$ satisfying the  two conditions in \cref{prop:new}.  We adopt the same notation as that in \cref{prop:new} and \Cref{sec:Torelli} throughout this  section. Let us denote by $n$ the largest non-negative number for $k$ so that $\tau_k$ in \eqref{iterated Kodaira2} is not trivial. By \cref{Deng}, $n>0$.   Following  \cite[Section 2.3]{Den18}, we construct Finsler metrics $F_1,\ldots,F_{n}$ on $T_Y(-\log D)$ as follows.  
By \eqref{iterated Kodaira2},  for each \(k=1,\ldots,n \), there exists a 
\begin{eqnarray*} 
\tau_k\colon {\Sym}^k  T_{Y}(-\log D) \lra  L^{-1}\otimes E^{p_0-k,q_0+k}.
\end{eqnarray*}
Then it follows from \Cref{bounded} that the (Finsler) metric \(h \) on \(  L^{-1}\otimes E^{p_0-k,q_0+k} \) induces  a Finsler metric \(F_k\) on \( T_Y(-\log D) \) defined as follows: for any \(e\in   T_{Y}(-\log D)_y \),
\begin{eqnarray}\label{k-Finsler}
F_k(e):=  h\left(\tau_k\left(e^{\otimes k}\right)\right)^{\frac{1}{k}}.
\end{eqnarray}
Let $C\subset \bC$ be any open subset of $\bC$.  For any holomorphic map \(\gamma:C\rightarrow U: =Y-D\), one has
\begin{align} \label{eq:differential}
d\gamma\colon T_{C}\lra \gamma^* T_U= \gamma^* T_{Y}(-\log D).
\end{align}
We	denote by \(\d_t:=\frac{\d}{\d t} \)  the canonical vector field  in \( C\subset \bC\),  and by  \(\bar{\d}_t:=\frac{\d}{\d \bar{t}}\) its conjugate.  
The Finsler metric \(F_k\) induces a continuous hermitian pseudo metric on \(C \), defined by
\begin{eqnarray}\label{seminorm}
\gamma^*F_k^2=\sqrt{-1}G_k(t) dt\wedge d\bar{t}.
\end{eqnarray}
Hence $G_k(t)=|\tau_k\big(d\gamma(\d_t)^{\otimes k}\big)|^{{2}/{k}}_{h}$,
where $\tau_k$ is defined in \eqref{iterated Kodaira2}.

By \cref{Deng}, there is a  Zariski open subset $U^\circ $ of $U$ such that $U^\circ \cap \mathbf{B}_+(L)=\varnothing$ and $\tau_1 $ is injective at any point of $U^\circ$.  We now fix any holomorphic map $\gamma\colon C\to U$ with  $\gamma(C)\cap U^\circ\neq \varnothing$.  By \Cref{bounded}, the metric $h$ for $L^{-1}\otimes E$ is smooth and positive definite over $U-\mathbf{B}_+(L)$. Hence $G_1(t)\not\equiv 0$.  Let $C^\circ $ be a (non-empty) open subset of $C$ whose  complement $C\setminus C^\circ$ is a \emph{discrete set} so that 
\begin{itemize}
	\item   $\gamma(C^\circ)\subset U^\circ$; 
	\item  for every $k=1,\ldots,n$,  either $G_k(t)\equiv 0$ on $C^\circ$, or $G_k(t)>0$ for every $t\in C^\circ$; 
	\item $\gamma'(t)\neq0$ for any $t\in C^\circ$; namely $\gamma|_{C^\circ}\colon C^\circ\to U^0$ is immersive everywhere. 
\end{itemize}
By the definition of $G_k(t)$, if $G_k(t)\equiv 0$ for some $k>1$, then $\tau_k(\d_t^{\otimes k})\equiv 0 $, where $\tau_k$ is defined in \eqref{iterated Kodaira2}. Note that one has $\tau_{k+1}(\d_t^{\otimes (k+1)})=\tilde{\theta}(\tau_k(\d_t^{\otimes k}))(\d_t)$,  where
$$
\tilde{\theta}=\mathds{1}_{L^{-1}}\otimes \theta\colon L^{-1}\otimes E\lra L^{-1}\otimes E\otimes \Omega_{Y}^1(\log D). 
$$
We thus conclude that $G_{k+1}(t)\equiv 0$. Hence there exists an $m$ with  $1\leq m\leq n$ so that the set $\{k\mid G_k(t)> 0$ $\mbox{over }C^\circ\}=\{1,\ldots,m\}$ and $G_{\ell}(t)\equiv 0$ for all $\ell=m+1,\ldots,n$.  From now on, \emph{all the computations} are made over $C^\circ$ if not specified.

Using the same computations as those in the proof of  \cite[Proposition 2.10]{Den18}, we have the following curvature formula.

\begin{thm}\label{thm:curvature}
	For $k=1,\ldots,m$, over $C^\circ$, one has
	\begin{align}  \label{eq:k=1}
	\frac{	\d^2\log G_1}{\d t\d \bar{t}}  	&\geq\Theta_{L,h_L}\left(\d_t,\bar{\d}_t\right) -  \frac{G_2^2}{G_1}    \quad  &\mbox{ if  }\ k=1, \\\label{eq:k>2}
	\frac{	\d^2\log G_k}{\d t\d \bar{t}}  	&\geq  \frac{1}{k}\left( \Theta_{L,h_L}\left(\d_t, \bar{\d}_t\right)+  \frac{G_k^k}{G^{k-1}_{k-1}}  - \frac{G^{k+1}_{k+1}}{G^k_k}    \right)  \quad  &\mbox{ if  }\ k>1. 
	\end{align} 
	Here we make the convention that $G_{m+1}\equiv 0$ and $\frac{0}{0}=0$. We  also write $\d_t$ $($resp.\ $\bar{\d}_t)$ for $d\gamma(\d_t)$ $($resp.\ ${d\gamma(\bar{\d}_t)})$ abusively, where $d\gamma$ is defined in \eqref{eq:differential}.   
\end{thm}

Let us mention that  in \cite[Equation (2.2.11)]{Den18}, we dropped the term 
$ \Theta_{L,h_L}(\d_t, \bar{\d}_t)$ in \eqref{eq:k>2}, though it can be easily seen from the proof of \cite[Lemma 2.7]{Den18}.  

We will follow  ideas in  \cite[Section 2.3]{Den18} (inspired  by \cite{TY14,BPW17,Sch17})  to  introduce a new Finsler metric $F$ on $T_Y(-\log D)$ by taking a convex sum of the form
\begin{align}
F:=\sqrt{\sum_{k=1}^{n}k\alpha_kF_k^2}, 
\end{align} 
where \(\alpha_1,\ldots,\alpha_n\in \mathbb{R}^+ \) are some constants  which will be fixed later.

For the above, for  \(\gamma\colon C\rightarrow  U \) with $\gamma(C)\cap U^\circ\neq \varnothing$, we write 
$$
\gamma^*F^2=\sqrt{-1}H(t)dt\wedge d\bar{t}.
$$
Then 
\begin{align}\label{eq:H}
H(t)=\sum_{k=1}^{n}  {{k\alpha_k}G_k}(t), 
\end{align} where \(G_k\) is defined in \eqref{seminorm}.  Recall that for $k=1,\ldots,m$,  $G_k(t)> 0$ for any $t\in C^\circ$.  

We first recall a computational lemma by Schumacher.

\begin{lem}[\!\protecting{\cite[Lemma 17]{Sch17}}]
  Let \(\alpha_j\) and $G_j$ be positive real numbers for \(j=1,\ldots,n \). Then 
	\begin{equation}\label{final}
          \sum_{j=2}^{n}\left(\alpha_{j}\frac{G_j^{j+1}}{G_{j-1}^{j-1}}-\alpha_{j-1}\frac{G_j^{j}}{G_{j-1}^{j-2}} \right)
          \geqslant \frac{1}{2}\left(-\frac{\alpha_1^3}{\alpha_2^2}G_1^2+\frac{\alpha_{n-1}^{n-1}}{\alpha_n^{n-2}}G_n^2+\sum_{j=2}^{n-1}\left(\frac{\alpha_{j-1}^{j-1}}{\alpha_j^{j-2}} -\frac{\alpha_{j}^{j+2}}{\alpha_{j+1}^{j+1}}\right)G_j^2 \right)
	\end{equation}   
\end{lem}

Now we are ready to compute the curvature of the Finsler metric $F$ based on \cref{thm:curvature}.

\begin{thm}\label{curvature estimate}  Fix a smooth K\"ahler metric $\omega$ on $Y$.	There exist  \emph{universal} constants \(0<\alpha_1<\ldots<\alpha_n\) and \(\delta>0\) such that for any holomorphic map $\gamma\colon C\to U=Y-D$ with $C$ an open subset of\, $\bC$ and $\gamma(C)\cap U^\circ\neq \varnothing$, one has
	\begin{align}\label{eq:estimatefinal}
	\hess \log|\gamma'(t)|_{F}^2\geq \delta\gamma^*\omega.
	\end{align} 
\end{thm}

\begin{proof}
	By \cref{Deng} and the assumption that $\gamma(C)\cap U^\circ\neq \varnothing$, we have $G_1(t)\not\equiv0$. 

We first recall a  result in   \cite[Lemma 2.9]{Den18};  we write its proof here as it is crucial in what follows.

	\begin{claim}\label{claim2}
		There is a universal constant $c_0>0$ $($\textit{i.e.}, it does not depend on $\gamma)$ so that $ \Theta_{L,h_L}(\d_t,\bar{\d}_t)\geq c_0G_1(t)$ for all $t\in C$.
	\end{claim}

        \begin{proof}[Proof of \Cref{claim2}]
		Indeed, by \Cref{estimate}, it suffices to prove that  
		\begin{eqnarray}\label{strict positive}
		\frac{|\d_t|^2_{\gamma^*(T)}}{\left|  \tau_1(d\gamma(\d_t))\right|_{h}^2}\geqslant c_0
		\end{eqnarray}
		for some $c_0>0$, where \(T\) is a K\"ahler current on $Y$, which is a smooth complete metric over $Y-D$ of Poincar\'e type.  It can be seen as a singular hermitian metric for $T_Y(-\log D)$. Hence for any admissible coordinate $(\cU;z_1,\ldots,z_n)$, one has
		$$
		\left|z_i\frac{\d}{\d z_i}\right|_T\sim (-\log |z_i|)^{-1}.
		$$ 
		On the other hand, by \cref{bounded}, one has
		$$
		\left|\tau_1\left(z_i\frac{\d}{\d z_i}\right)\right|_{h}\lesssim C\cdot \prod_{i=1}^{\ell}|z_i|^{\ep}
		$$
		for some constant $\ep>0$. Hence one has
		$\tau_1^*h\lesssim T$.
                Since \(Y\) is compact,    there exists a constant  \(c_0>0\)  such that
		 $
		T\geqslant c_0\tau_1^*h
		 $.
		Therefore,
		\[
		\frac{|\d_t|^2_{\gamma^*T}}{\left|  \tau_1(d\gamma(\d_t))\right|_{h}^2}=\frac{|\d_t|^2_{\gamma^*T}}{|\d_t|^2_{\gamma^*\tau_1^*h}}\geq c_0.
		\] 
	 	Hence  \eqref{strict positive} holds for any $\gamma\colon C\to U$ with $\gamma(C)\cap U^\circ\neq \varnothing$. The claim is proved.
	\end{proof} 
		
	By \cite[Lemma 8]{Sch12}, 
	\begin{eqnarray}\label{calculus}
	\sqrt{-1}\d\bar{\d}\log \left(\sum_{j=1}^{n}{j\alpha_j}G_j \right)\geqslant \frac{\sum_{j=1}^{n}j\alpha_jG_j\sqrt{-1}\d\bar{\d}\log G_j}{\sum_{i=1}^{n}j\alpha_jG_i}.
	\end{eqnarray}
	Substituting \eqref{eq:k=1} and \eqref{eq:k>2} into \eqref{calculus}, and observing the convention that $\frac{0}{0}=0$, we obtain
	\begin{align*}\nonumber
	\frac{\d^2\log H(t)}{\d t\d \bar{t}}&\geq  \frac{1}{H}\left(-\alpha_1G_2^2+ \sum_{k=2}^{n}\alpha_k\left(\frac{G_k^{k+1}}{G^{k-1}_{k-1}}  - \frac{G^{k+1}_{k+1}}{G^{k-1}_{k}}   \right)\right)+\frac{\sum_{k=1}^{n}\alpha_kG_k}{H} \Theta_{L,h_L}(\d_t,\bar{\d}_{t})\\
	&= \frac{1}{H}\left( \sum_{j=2}^{n}\left(\alpha_j\frac{G_j^{j+1}}{G_{j-1}^{j-1}}-\alpha_{j-1}\frac{G_j^{j}}{G_{j-1}^{j-2}} \right) \right)+\frac{\sum_{k=1}^{n}\alpha_kG_k}{H} \Theta_{L,h_L}(\d_t,\bar{\d}_{t})\\\nonumber
	&\stackrel{\eqref{final}}{\geq}   \frac{1}{H}\left(-\frac{1}{2}\frac{\alpha_1^3}{\alpha_2^2}G_1^2+\frac{1}{2}\sum_{j=2}^{n-1}\left(\frac{\alpha_{j-1}^{j-1}}{\alpha_{j}^{j-2}}-\frac{\alpha_j^{j+2}}{\alpha_{j+1}^{j+1}} \right)G_j^2+\frac{1}{2}\frac{\alpha_{n-1}^{n-1}}{\alpha_n^{n-2}}G_n^2  \right) \\\nonumber
	&\hphantom{\geq +}\;+\frac{\sum_{k=1}^{n}\alpha_kG_k}{H} \Theta_{L,h_L}(\d_t,\bar{\d}_{t})\\
	&\stackrel{\Cref{claim2}}{\geq} \frac{1}{H}\left(\frac{\alpha_1}{2}\left(c_0-\frac{\alpha_1^2}{\alpha_2^2}\right)G_1^2+\frac{1}{2}\sum_{j=2}^{n-1}\left(\frac{\alpha_{j-1}^{j-1}}{\alpha_{j}^{j-2}}-\frac{\alpha_j^{j+2}}{\alpha_{j+1}^{j+1}} \right)G_j^2+\frac{1}{2}\frac{\alpha_{n-1}^{n-1}}{\alpha_n^{n-2}}G_n^2  \right) \\\nonumber
	&\hphantom{\geq Clan +}\;+\frac{1}{H}\left(\frac{1}{2}\alpha_1G_1+\sum_{k=2}^{n}\alpha_kG_k\right) \Theta_{L,h_L}(\d_t,\bar{\d}_{t}).
	\end{align*}
	One can take \(\alpha_1=1 \) and choose the further \(\alpha_j>\alpha_{j-1} \) inductively   so that 
	\begin{align}\label{eq:alpha}
	c_0-\frac{\alpha_1^2}{\alpha_2^2}>0, \quad \frac{\alpha_{j-1}^{j-1}}{\alpha_{j}^{j-2}}-\frac{\alpha_j^{j+2}}{\alpha_{j+1}^{j+1}} >0 \quad \forall \ j=2,\ldots,n-1. 
	\end{align}
	Hence 
	$$
	\frac{\d^2\log H(t)}{\d t\d \bar{t}}\geq   \frac{1}{H}\left(\frac{1}{2}\alpha_1G_1+\sum_{k=2}^{n}\alpha_kG_k\right) \Theta_{L,h_L}(\d_t,\bar{\d}_{t}) \stackrel{\eqref{eq:H}}{\geq}\frac{1}{n}\Theta_{L,h_L}(\d_t,\bar{\d}_{t}) 
	$$ 
	over $C^\circ$. 
	By \Cref{estimate}, this implies that 
	\begin{align}\label{eq:curvature esti}
	\hess \log |\gamma'|_{F}^2=  \hess \log H(t)\geq \frac{1}{n}\gamma^*\sn\Theta_{L,h_L}\geq \delta\gamma^*\omega
	\end{align}
	over $C^\circ$ for some positive constant $\delta$ which does not depend on $\gamma$.  Since $|\gamma'(t)|_F^2$ is continuous and locally bounded from above over   $C$,  by the extension theorem of  subharmonic function, \eqref{eq:curvature esti}  holds over the whole~$C$.   Since $c_0>0$ is a constant which does not depend on $\gamma$, so are $\alpha_1,\ldots,\alpha_n$ by  \eqref{eq:alpha}.  The theorem is thus proved.
\end{proof}

As a summary of the results in this subsection, we obtain the following theorem.

\begin{thm}\label{thm:uniform}
Let $(E=\oplus_{p+q=\ell}E^{p,q},\theta)$ be a system of log Hodge  bundles on a compact K\"ahler log pair $(Y,D)$ satisfying the  two conditions in \cref{prop:new}.  Then there are a Finsler metric $h$ on $T_Y(-\log D)$ which is \emph{positive definite} on a dense Zariski open subset $U^\circ$ of\, $U:=Y- D$ and a smooth K\"ahler form $\omega$ on $Y$ such that   for any holomorphic map $\gamma\colon C\to U$  from any open subset $C$ of\, $ \bC$ with $\gamma(C)\cap U^\circ\neq \varnothing$, one has
	\begin{align} \label{eq:desired}
	\hess \log |\gamma'|_h^2\geq \gamma^*\omega. 
	\end{align} 
\end{thm}

\begin{proof} [Proof of \cref{thmx:uniform}]
	\cref{prop:new} together with \cref{thm:uniform} imply \cref{thmx:uniform}. 
	\end{proof}

\section{Big Picard theorem and algebraic hyperbolicity}

\subsection{Algebraic  and Picard hyperbolicity}
In \cref{def:notion}, we have seen the definition of  \emph{algebraic hyperbolicity} for a compact complex manifold $X$, which was   introduced by Demailly in \cite[Definition 2.2]{Dem97}. He proved in \cite[Theorem 2.1]{Dem97} that $X$ is algebraically hyperbolic if it is  Kobayashi hyperbolic.  The notion of algebraic hyperbolicity was generalized to log pairs   by Chen \cite{Che04}.
\begin{dfn}[Algebraic hyperbolicity]\label{def:DC}
	Let $(X,D)$ be a compact K\"ahler log pair.  For  any reduced irreducible curve $C\subset X$ such that $C\not\subset D$,  we denote by $i_X(C,D)$ the number of distinct
	points in the set $\nu^{-1}(D)$, where $\nu\colon \tilde{C}\to C$
	is  the normalization of $C$.  The log pair $(X,D)$ is \emph{algebraically hyperbolic}  if there is a smooth K\"ahler metric $\omega$ on $X$ such that
	$$
	2g(\tilde{C})-2+i(C,D)\geq {\deg}_{\omega}C:=\int_C\omega
	$$
	for all curves $C\subset X$ as above.
\end{dfn} 

Note that   $2g(\tilde{C})-2+i(C,D)$    depends only on the complement $C-D$.
Hence the above notion of hyperbolicity also makes sense for quasi-projective manifolds: we say that a quasi-projective manifold $U$ is algebraically hyperbolic if it has a log compactification $(X,D)$  which is  algebraically hyperbolic. 

However, unlike Demailly's theorem, it is unclear to us that Kobayashi hyperbolicity or Picard hyperbolicity of $X-D$ will imply algebraic hyperbolicity of $(X,D)$. In \cite{PR07}, Pacienza--Rousseau  proved that if $X-D$ is hyperbolically embedded into $X$, the log pair $(X,D)$ (and thus $X-D$) is algebraically hyperbolic.

Before we prove that  \cref{def:Picard} does not depend on the compactification of $U$, we will need  the following proposition,  which is   a consequence of the deep extension theorem of meromorphic maps by Siu \cite{Siu75}. The \emph{meromorphic map} in this paper is defined in the sense of Remmert, and we refer the reader to  \cite[p. 243]{KG02} for the precise definition.

\begin{proposition} \label{extension theorem}
	Let $Y^\circ$ be a Zariski open subset of a compact K\"ahler manifold $Y$. Assume that $Y^\circ$ is Picard hyperbolic. Then any   holomorphic map $f\colon\Delta^p\times (\Delta^*)^q\to Y^\circ$ extends to a meromorphic map $\overline{f}\colon\Delta^{p+q}\dashrightarrow Y$.  In particular, any holomorphic map $g$ from a Zariski open subset $X^\circ$ of a compact complex manifold $X$ to $Y^\circ$ extends to a meromorphic map from $X$ to $Y$.
\end{proposition}

\begin{proof} 
By \cite[Theorem 1]{Siu75}, any meromorphic map from a Zariski open subset $Z^\circ$ of a complex manifold $Z$ to a compact K\"ahler manifold $Y$ extends to a meromorphic map from $Z$ to $Y$ provided that the codimension of $Z-Z^\circ$ is at least $2$.   The complement $\Delta^p\times (\Delta^*)^q$ in $\Delta^{p+q}$ is a simple normal crossing divisor $D$. We  remove a subvariety $Z\subset \Delta^{p+q}$ of codimension at least $2$ with $D-Z$  smooth.  Then  any point $x\in D-Z$ has an open  neighborhood $\Omega_x\subset \Delta^{p+q}-Z$ which is isomorphic to $\Delta^{p+q-1}\times \Delta^*$.  It then suffices to prove the extension theorem for any holomorphic map  $f\colon\Delta^{r}\times \Delta^*\to Y^\circ$.

By the assumption that $Y^\circ$ is Picard hyperbolic,  for any $z\in \Delta^r$, the holomorphic map  $f|_{\{z\}\times \Delta^*}\colon\{z\}\times \Delta^*\to Y^\circ$ can be extended to a holomorphic map from $\{z\}\times \Delta$ to $Y$. It then follows from  \cite[p.442,  ($\ast$)]{Siu75} that $f$ extends to a meromorphic map  $\overline{f}\colon\Delta^{r+1}\dashrightarrow Y$. This proves the first part of the proposition. To prove the second part, we first apply the Hironaka theorem on resolution of singularities  to assume that $X-X^\circ$ is a simple normal crossing divisor on $X$. Then any point $x\in X-X^\circ$ has an open  neighborhood $\Omega_x$ which is isomorphic to $\Delta^{p+q}$ so that $X^\circ\simeq \Delta^p\times (\Delta^*)^q$ under this isomorphism. The above arguments show that $g|_{\Omega_x\cap X^\circ}$ extends to a meromorphic map from $\Omega_x$ to $Y$, and thus  $g$ can be extended to  a meromorphic map from $X$ to $Y$. The proposition is proved.
\end{proof} 

Let us prove that \cref{def:Picard} does not depend on the compactification of $U$. This independence also implies the following result. 

\begin{lem}\label{lem:well-def}
	Let $U$ be a Zariski open subset of a compact  K\"ahler manifold $Y$.  If any holomorphic map $f\colon\Delta^*\to U$ extends to  $\bar{f}\colon\Delta\to Y$, then $Y$ is bimeromorphic to any other compact K\"ahler manifold $Y'$ which contains $U$ as a Zariski open set. In particular, $f\colon\Delta^*\to U$ also extends to a holomorphic map $\Delta\to Y'$.
\end{lem}

\begin{proof}
 	By blowing up $Y-U$ and $Y'-U$, we can assume that both $Y-U$ and $Y'-U$ are simple normal crossing divisors. By the same arguments as those in the proof of \cref{extension theorem}, the identity map of $U$ extends to  meromorphic maps $a\colon Y\dashrightarrow Y'$ and $b\colon Y'\dashrightarrow Y$. Note that $a\circ b|_{U}$ and $b\circ a|_U$ are identity maps. Hence $Y$ and $Y'$ are bimeromorphic.  Composing $b$ with $\bar{f}$, one obtains the desired extension $\Delta\to Y'$ of $f\colon\Delta^*\to U$ in $Y'$.
	\end{proof}

By  Chow's theorem, \cref{extension theorem} in particular gives an alternative proof of the fact that a Picard hyperbolic variety is moreover Borel hyperbolic, proven in \cite[Corollary 3.11]{JK18}.

\subsection{Proof of Theorem A}
This subsection is devoted to the proof of \cref{main}. We first recall the following criteria for Picard hyperbolicity established in \cite{DLSZ},
whose proof is   Nevanlinna-theoretic. 

\begin{thm}[\!\protecting{\cite[Theorem A]{DLSZ}}]\label{thm:criteria} 
	Let $Y$ be a projective manifold,   and let $D$ be a simple normal crossing divisor on $Y$. Let $f\colon\Delta^*\to Y- D$  be  a  holomorphic map. Assume that there is  a $($possibly degenerate$)$ Finsler metric $h$ of $T_Y(-\log D)$ such that $  |f'(t)|_h^2\not\equiv 0$ and 
	\begin{align} \label{eq:condition}
	\hess\log |f'(t)|_h^2\geq f^*\omega
	\end{align}
	for some smooth K\"ahler metric $\omega$ on $Y$.   Then $f$ extends to a holomorphic map $\overline{f}\colon\Delta\to Y$.
\end{thm}

We will combine \cref{thm:criteria} with \cref{thmx:uniform} to  prove  \cref{main}.

\begin{proof}[Proof of \cref{main}]
	By \cref{thmx:uniform}, there exist finitely many   compact K\"ahler log pairs $\{(X_i,D_i)\}_{i=0,\ldots,N}$ so that the following hold: 
	\begin{enumerate}
		\item There are   morphisms $\mu_i\colon X_i\to Y$ with $\mu_i^{-1}(D)=D_i$ so that each $\mu_i\colon X_i\to \mu_i(X_i)$ is a birational morphism and  $X_0=Y$ with $\mu_0=\mathds{1}$.
		\item \label{Finsler} There are smooth Finsler metrics $h_i$ for $T_{X_i}(-\log D_i)$ which is positive definite over a Zariski open subset $U_i^{\circ}$ of $U_i:=X_i-D_i$.
		\item \label{isomorphism} The restriction $\mu_i|_{U^\circ_i}\colon U^\circ_i\to \mu_i(U^\circ_i)$ is an isomorphism.
		\item \label{curvature}There are   smooth K\"ahler metrics $\omega_i$ on $X_i$ such that for any holomorphic map $\gamma\colon C\to U_i$ with $C$ an open subset of $\bC$ and $\gamma(C)\cap U^\circ_i\neq 0$, one has
		\begin{align}\label{eq:uniform}
		\hess \log |\gamma'|_{h_i}^2\geq \gamma^*\omega_i.   
		\end{align}
		\item  \label{stratum} For any $i\in \{0,\ldots,N\}$, either $		\mu_i( U_i)-\mu_i(U_i^\circ)$ is zero-dimensional, or  there exists an $I\subset  \{0,\ldots,N\}$ so that 
		$$
		\mu_i( U_i)-\mu_i(U_i^\circ)\subset \cup_{j\in I}\mu_j(X_j). 
		$$
	\end{enumerate} 
        Let us explain how to construct these log pairs. By the assumption, there is a $\bC$-PVHS  on $Y-D$ so that each fiber of the  period map  is zero-dimensional. In particular, the period map is    generically immersive. We then apply \cref{thmx:uniform} to construct a Finsler metric on $T_Y(-\log D)$ which is positive definite over some Zariski open subset $U^\circ$ of $U=Y-D$ with the desired curvature property \eqref{eq:desired}.  Set $X_0=Y$, $\mu_0=\mathds{1}$ and $U_0^\circ=U^\circ$.  Let $Z_1,\ldots,Z_m$ be all irreducible subvarieties of $Y-U^\circ$ which are not components of $D$. Then $Z_1\cup \ldots \cup Z_m\supset U\setminus U^\circ$. For each $i$, we take a desingularization $\mu_i\colon X_i\to Z_i$   so that $D_i:=\mu_i^{-1}(D)$ is a simple normal crossing divisor in $X_i$.  We pull back the  $\bC$-PVHS  to $U_i=X_i-D_i$   via $\mu_i$. Then   its period map is still generically immersive. We then apply \Cref{thmx:uniform} to construct the desired Finsler metrics in \cref{curvature} for $T_{X_i}(-\log D_i)$. We iterate this construction, and since at each step the dimension of $X_i$ is strictly decreasing, this algorithm stops after finitely many steps.

(i)~  We will first prove that $U$ is Picard hyperbolic. Fix any holomorphic map  $f\colon \Delta^*\to U$.    If $f(\Delta^*)\cap  U^\circ_0\neq \varnothing$, then $|f'(t)|_{h_0}\not\equiv0$ by \cref{Finsler}.  By  \cref{curvature},  there  is a  smooth K\"ahler metric $\omega_0$ on $X_0$ such that 
			\begin{align*}  
			\hess\log |f'(t)|_{h_0}^2\geq f^*\omega_0. 
		\end{align*} 
	We now apply \cref{thm:criteria} to   conclude that $f$ extends to a holomorphic map $\overline{f}\colon\Delta\to X_0=Y$.  
		
		Now assume $f(\Delta^*)\cap \mu_0(U^\circ_0)=\varnothing$. By \cref{stratum},  there exists an $I_0\subset  \{0,\ldots,N\}$ so that 
		$$
		f(\Delta^*)\subset 	 \mu_0( U_0)-\mu_0(U_0^\circ)\subset \cup_{j\in I_0}\mu_j(X_j). 
		$$
		Since the $\mu_j(X_j)$ are all irreducible, there exists a $k \in I_0$ so that $f(\Delta^*)\subset \mu_k(X_k)$. Note that $U_k:=\mu_k^{-1}(U)$. Hence $f(\Delta^*)\subset \mu_k(U_k)$. If $f(\Delta^*)\cap \mu_k(U_k^\circ)\neq \varnothing$, then by \cref{isomorphism},  $f(\Delta^*)$ is not contained in the exceptional set of $\mu_k$.  Hence $f$ can be lifted to $f_k\colon\Delta^*\to U_k$, so that $\mu_k\circ f_k=f$ and $f_k(\Delta^*)\cap U_k^\circ\neq\varnothing$. By \cref{thm:criteria} and \cref{curvature} again, we conclude that $f_k$ extends to a holomorphic map $\overline{f}_k\colon\Delta\to X_k$. Hence $\mu_k\circ\overline{f}_k$ extends $f$. If $f(\Delta^*)\cap \mu_k(U_k^\circ)=\varnothing$, we apply \cref{stratum} to iterate the above arguments, and after finitely many steps,  there exists an $X_i$ so that $f(\Delta^*)\subset \mu_i(U_i)$ and $f(\Delta^*)\cap \mu_i(U_i^\circ)\neq \varnothing$. By \cref{isomorphism},  $f$ can be lifted to $f_i\colon\Delta^*\to U_i$ so that $\mu_i\circ f_i=f$ and $f_i(\Delta^*)\cap U_i^\circ\neq \varnothing$. By  \cref{thm:criteria} and \cref{curvature} again, $f_i$ extends to the origin, and so does $f$. This proves the Picard hyperbolicity of $U=Y-D$.

(ii)~ Let us prove the algebraic hyperbolicity of $U$ in as similar vein as   \cite[Proof of Theorem D]{DLSZ}. Fix any reduced and irreducible curve $C\subset Y$ with $C\not\subset D$. By the above arguments, there exists an $i\in \{0,\ldots,N\}$ so that $C\subset \mu_i(X_i)$  and $C\cap \mu_i(U_i^\circ)\neq \varnothing$. Let $C_i\subset X_i$ be the strict transform of $C$  under $\mu_i$.  By \cref{isomorphism},  $h_i|_{C_i}$ is not identically equal to zero. 
		
		Denote by $\nu_i\colon\tilde{C}_i\to C_i\subset X_i$
  the normalization of $C_i$, and set $P_i:=(\mu_i\circ\nu_i)^{-1}(D)=\nu_i^{-1}(D_i)$. Then
  $$
  d\nu_i\colon T_{\tilde{C}_i}(-\log P_i)\lra \nu_i^*T_{X_i}(-\log D_i)
  $$
induces   a (non-trivial) hermitian pseudo metric $\tilde{h}_i:=\nu_i^*h_i$  over $T_{\tilde{C}_i}(-\log P_i)$. By \eqref{eq:uniform}, the \emph{curvature current} of $\tilde{h}_i^{-1}$ on $\tilde{C}_i$ satisfies
		$$
		\frac{\sn}{2\pi}\Theta_{\tilde{h}_i^{-1}}(K_{\tilde{C}_i}(\log P_i))\geq \nu_i^*\omega_i.
		$$
		Hence
		$$
		2g(\tilde{C}_i)-2+i(C,D)=\int_{\tilde{C}_i}  \frac{\sn}{2\pi}\Theta_{\tilde{h}_i^{-1}}\left(K_{\tilde{C}_i}(\log P_i)\right)\geq \int_{\tilde{C}_i} \nu_i^*\omega_i.
		$$
		Fix a K\"ahler metric $\Omega_Y^1$ on $Y$. Then there is a constant $\ep_i>0$ so that $\omega_i\geq\ep_i\mu_i^*\Omega_Y^1$. We thus have
		$$
		2g(\tilde{C}_i)-2+i(C,D)  \geq \ep_i\int_{\tilde{C}_i} (\mu_i\circ \nu_i)^*\Omega_Y^1=\ep_i\deg_{\Omega_Y^1}C
		$$
	for $\mu_i\circ\nu_i\colon \tilde{C}_i\to C$  the normalization of $C$. 
		Set $\ep:=\inf_{i=0,\ldots,N}\ep_i$. Then we conclude that for any  reduced and irreducible curve $C\subset Y$ with $C\not\subset D$, one has
		$$
		2g(\tilde{C})-2+i(C,D)  \geq \ep\deg_{\Omega_Y^1}C, 
		$$
		where $\tilde{C}\to C$ is its normalization. This shows the algebraic hyperbolicity of  $U$. The proof of the theorem is accomplished.
\end{proof}

Let us mention that the idea of using Finsler metrics to prove the hyperbolicity in the above theorem was inspired by the work of To--Yeung in \cite{TY14}.

 \begin{rem}\label{JL}
Let $U$ be a  quasi-projective manifold  admitting an integral variation of Hodge structures   whose  period map is quasi-finite. In \cite[Theorem 4.2]{JL19}, Javanpeykar--Litt proved that  $U$  is \emph{weakly bounded} in the sense of Kov\'acs--Lieblich \cite[Definition 2.4]{KL10} (which is weaker than algebraic hyperbolicity). Though not mentioned explicitly,   their proof of \cite[Theorem 4.2]{JL19} implicitly  shows that such a $U$ is also algebraically hyperbolic when the   local monodromies of   the $\bC$-PVHS at infinity are unipotent.  
 Their proof is based on the work   \cite{BBT18}  as well as the Arakelov-type inequality for Hodge bundles by Peters \cite{Pet00}.   
 \end{rem}

We end this section with the following remark.

\begin{rem}
Let $(E,\theta)$ be the system of log Hodge  bundles on a log pair $(Y,D)$ as that in \cref{thm:uniform}. One can also use the idea by Viehweg--Zuo  \cite{VZ02} in constructing their \emph{Viehweg--Zuo} sheaf (based on the negativity of kernels of Higgs fields by Zuo \cite{Zuo00}) to prove a weaker result than \cref{thm:uniform}: for any holomorphic  map $\gamma\colon C\to U$  from any open subset $C$ of $ \bC$ with $\gamma(C)\cap U^\circ\neq \varnothing$, there exist a Finsler metric $h_C$ of $T_Y(-\log D)$ (depending on $C$) and a K\"ahler metric $\omega_C$ for $Y$ (also depending on $C$) so that $|\gamma'(t)|_h^2\not\equiv 0$ and
	\begin{align*} 
	\hess \log |\gamma'|_{h_C}^2\geq \gamma^*\omega_C. 
	\end{align*}   
	It follows from our proof of \cref{main} that one can also combine  \cref{thm:criteria} with this result, which is only weaker in appearance,  to prove \cref{main}. The more general result \cref{thmx:uniform} will be used to prove \Cref{pseudo Kobayashi} in the next section.
\end{rem} 

\section{Hyperbolicity for the compactification after a finite \'etale cover}\label{sec:strong}  
 
In this section, we will prove \cref{thmx:refined,Picard symmetric}. We first prove the following theorem.

\begin{thm}\label{strong}
 	Let $U$ be a quasi-compact K\"ahler manifold.  Assume that    there is a  $\bC$-PVHS   over $U$  whose period map is   immersive at one point.   Then there are a finite \'etale cover $\tilde{U}\to U$ together with a compact K\"ahler compactification $X$ of\, $\tilde{U}$ and a proper  Zariski closed subvariety $Z\subsetneq X$   so that
 	\begin{thmlist}
 		\item \label{general type}  the variety $X$    is  of  general type; 
 		\item  \label{pseudo Kobayashi} the variety $X$  is  Kobayashi hyperbolic modulo $Z$; 
 		\item  \label{pseudo Picard} the variety $X$  is   Picard hyperbolic modulo $Z$; 
 		\item  \label{pseudo algebraic} the variety $X$  is  algebraically hyperbolic modulo $Z$.
 	\end{thmlist}
 \end{thm}  
 
Let us briefly explain the idea of the proof of \cref{strong} and some related results. Let $Y$ be a compact K\"ahler manifold compactifying $U$ with $D:=Y-U$ a simple normal crossing divisor. 
By   \cref{prop:new}, there is a  special system of log Hodge  bundles $(E,\theta):=(\oplus_{p+q=\ell}E^{p,q},\oplus_{p+q=\ell}\theta_{p,q})$ on $(Y,D)$ satisfying the properties therein. We divide the proof into four steps.
\begin{enumerate}
  	\item The first step is devoted  to constructing a compact K\"ahler log pair $(X,\tilde{D})$ and a generically finite surjective log morphism $\mu\colon (X,\tilde{D})\to (Y,D)$   which is \'etale over $U$ so that for each irreducible component $\tilde{D}_i$ of $\tilde{D}$,
	\begin{itemize}
		\item either ${\ord}_{\tilde{D}_i}(\mu^*D)\gg 1$, 
		\item or the local monodromy of the pull-back $\bC$-PVHS over $\tilde{U}$ around $\tilde{D}_i$ is trivial. 
	\end{itemize} To find this   $\mu$, we apply the well-known result that that monodromy group of a $\bC$-PVHS is residually finite and use the Cauchy argument principle to show the high ramification over irreducible components of $\tilde{D}$ around which the local monodromies are not trivial. Let us mention that this step is quite different from those in  \cite{Nad89,Rou16,Bru16,Cad18} for the   hyperbolicity of compactifications of quotients of bounded symmetric domains by a torsion-free lattice, as they all applied Mumford's work on  toroidal compactifications of quotients of bounded symmetric domains \cite{Mum77} so that ${\ord}_{\tilde{D}_i}(\mu^*D)\gg 1$ for all $\tilde{D}_i$. In general, we are not sure that such a covering  can be found in our case. 
	\item  The second step is to construct a new system of log Hodge bundles  $(G=\oplus_{p+q=\ell}G^{p,q}, \eta)$ over  $(X,\tilde{D})$ which is the canonical extension of the pull-back of the $\bC$-PVHS via $\mu$. This system of log Hodge  bundles $(G=\oplus_{p+q=\ell}G^{p,q}, \eta)$ on $(X,\tilde{D})$  satisfies the two conditions in \cref{prop:new}. Moreover, some   $G^{p_0,\ell-p_0}$ contains   $\tilde{L}\otimes \cO_X(\ell D_X)$ with $\tilde{L}$ a big line bundle. Here   $D_X$ is the sum of irreducible components of $\tilde{D}$ around which the local monodromies of the pull-back $\bC$-PVHS are not trivial (hence $\mu$ is highly ramified over $D_X$).  Note that $(G,\eta)$ has singularities along $D_X$ instead of $\tilde{D}$ since the pull-back $\bC$-PVHS extends across the components where the local monodromies are trivial (see \eqref{eq:regular}.)
	\item The third step is to prove \cref{general type}. We start with $G^{p_0,\ell-p_0}$ and iterate the Higgs field $\eta$, ending at finitely many steps. By  the negativity of the kernel of $\tilde{\theta}$,   $\tilde{L}\otimes \cO(\ell D_X)\subset G^{p_0,\ell-p_0}$, and \eqref{eq:regular}, we can construct  an ample sheaf contained in some symmetric differential ${\Sym}^\beta \Omega_X^1$ (rather than on ${\Sym}^\beta \Omega_X^1(\log \tilde{D})$!). It follows from a celebrated theorem of Campana--P\u{a}un \cite{CP19} that $X$ is of general type. Let us mention that this idea of iterating Higgs fields to their kernels, originally due to Viehweg--Zuo \cite{VZ02}, has been  used by Brunebarbe in \cite{Bru16},   in which he proved similar results for quotients of bounded symmetric domains by arithmetic groups.  There are also some similar results for \emph{quotients of bounded domains} by Boucksom--Diverio \cite{BD18} and Cadorel--Diverio--Guenancia \cite{CDG19}.
	\item The last step is to prove \cref{pseudo Kobayashi}--\cref{pseudo algebraic}. We use the above system of log Hodge  bundles $(G,\eta)$ and ideas in \Cref{sec:construction} to construct a Finsler metric $F$ on $T_X$ (rather than $T_X(-\log D)$!) due to the extra positivity $\tilde{L}\otimes \cO(\ell D_X)\subset G^{p_0,\ell-p_0}$. Such a metric $F$ is generically positive and has holomorphic sectional curvature bounded from above by a negative constant by \cref{curvature estimate}. By the Ahlfors--Schwarz lemma, we conclude that $X$ is  Kobayashi hyperbolic modulo a proper closed subvariety, and by \cref{thm:criteria}, the  Picard hyperbolicity  modulo a proper subset of $X$ follows. Let us mention that Rousseau \cite{Rou16} has proved a similar result for hermitian symmetric spaces, which was later refined by Cadorel \cite{Cad18}. Their methods use Bergman metrics for bounded symmetric domains instead of Hodge theory.
\end{enumerate} 

Now we start the detailed proof of \cref{strong}.

\begin{proof}[Proof of \cref{strong}]   
By   \cref{prop:new}, there is a system of log Hodge  bundles $(E,\theta)\!\!=\!\!(\oplus_{p+q=\ell}E^{p,q},\oplus_{p+q=\ell}\theta_{p,q})$  over $(Y,D)$ satisfying the two conditions therein. In particular, there are a big line bundle $L$ on $Y$ and an inclusion $L\subset E^{p_0,\ell-p_0}$ for some $0\leq p_0\leq \ell$. Pick $m\gg 1$ so that $L-\frac{\ell+1}{m}D$ is   a big $\mathbb{Q}$-line bundle.

\medskip

\noindent {\it Step 1a}.  
Fix a base point $y\in U:=Y-D$. Let us denote by $\rho\colon\pi_1(U,y)\to GL(r,\bC)$ the monodromy representation of the corresponding $\bC$-PVHS and denote by $\Gamma:=\rho(\pi_1(U, y))$ its monodromy group, which is a finitely generated linear group, hence residually finite by a theorem of Malcev \cite{Mal40}. Let us cover $Y$ by finitely many admissible coordinate systems
$$\left\{\left(\mathcal{U}_{\alpha}; z_1^{(\alpha)},\ldots,z^{(\alpha)}_d\right) \right\}_{\alpha\in S},
$$
where $S$ is a finite set, so that
$$
D\cap \cU_{\alpha}=\left(z_1^{(\alpha)}\cdots z_{k_\alpha}^{(\alpha)}=0\right).
$$
Write $\cU_{\alpha}^*:=\cU_{\alpha}-D$. 
The fundamental group $\pi_1(\mathcal{U}_\alpha^*, y_\alpha)\simeq \pi_1((\Delta^*)^{k_\alpha}\times \Delta^{d-k_\alpha},y_\alpha)\simeq \bZ^{k_\alpha}$   is abelian. Pick a base point $y_\alpha\in \cU_\alpha^*$. Let $e^{(\alpha)}_1,\ldots,e^{(\alpha)}_{k_\alpha}$ be the generators of $\pi_1((\Delta^*)^{k_\alpha}\times \Delta^{d-k_\alpha},y_\alpha)$; namely, $e^{(\alpha)}_i$ is a clockwise loop around the origin in the $\supth{i}$ factor $\Delta^*$. Pick a path $h_\alpha\colon[0,1]\to Y-D$ connecting $y_{\alpha}$ with $y$, and denote by $\gamma^{(\alpha)}_i\in \pi_1(Y-D, y)$ the equivalent class of the loop $h_\alpha^{-1}\cdot e_i^{(\alpha)}\cdot h_\alpha$. Set $T^{(\alpha)}_i:=\rho(\gamma^{(\alpha)}_i)$.   Clearly, $T_1^{(\alpha)},\ldots, T_{\alpha_k}^{(\alpha)}$ commute pairwise.  

Let $\mathfrak{S}\subset \Gamma$ be the finite subset defined by  
\begin{align}\label{eq:choice}
\left\{ \left(T_{1}^{(\alpha)}\right)^{\ell_1}\cdots\left(T_{k_\alpha}^{(\alpha)}\right)^{\ell_{k_\alpha}}  \mid  \alpha\in S,     0\leq \ell_i<m  \right\}, 
\end{align}
where  $m$ is the integer chosen at the beginning. It follows from the definition of a  residually finite group that there is a normal subgroup $\tilde{\Gamma}$ of $ \Gamma$ with finite index   so that 
\begin{align}\label{eq:choice2}
\mathfrak{S}\cap \tilde{\Gamma}=\{0\}.
\end{align}
Then $\rho^{-1}(\tilde{\Gamma})$ is a normal subgroup of $\pi_1(U,y)$  with finite index. Let $\nu\colon\tilde{U}\to U$ be the finite \'etale cover of $U$ so that for the induced map of the fundamental group
 $\nu_*\colon \pi_1(\tilde{U}, x)\to \pi_1(U, y)$, its image is 
  $\rho^{-1}(\tilde{\Gamma})$. Here $x\in \tilde{U}$ with $\mu(x)=y$. We consider $\pi_1(\tilde{U}, x)$ as a subgroup of $\pi_1(U, y)$ of finite index. Since the monodromy representation of the pull-back  of the $\bC$-PVHS on $\tilde{U}$ is the restriction 
  \[
  \rho|_{\pi_1(\tilde{U}, x)}\colon \pi_1(\tilde{U}, x)\lra GL(r, \bC),
  \] 
  its monodromy group is thus $\tilde{\Gamma}$.

\medskip

\noindent {\it Step 1b}.
Note that $U$ is quasi-projective. Hence $\tilde{U}$ is also quasi-projective.  Let us take a smooth projective compactification $X$ of $\tilde{U}$ with $\tilde{D}:=X-\tilde{U}$ simple normal crossing so that $\nu\colon\tilde{U}\to U$ extends to a log morphism $\mu\colon(X,\tilde{D})\to (Y,D)$. Write $\tilde{D}=\sum_{j=1}^{n}\tilde{D}_j$, where the $\tilde{D}_j$ are irreducible components of $\tilde{D}$.   

\begin{claim}\label{claim}
For each $j=1,\ldots,n$,   one has
\begin{itemize}
	\item either  ${\ord}_{\tilde{D}_j}(\mu^*D) \geq m,$ 
	\item or the local monodromy group of the pull-back\, $\bC$-PVHS around $\tilde{D}_j$ is trivial.
\end{itemize} 
\end{claim}

\begin{proof}[Proof of \cref{claim}]
Since $\{\cU^{(\alpha)}\}_{\alpha \in S}$ covers $D$, there is an $\alpha\in S$ so that for the admissible coordinate system $(\mathcal{U}^{(\alpha)}; z^{(\alpha)}_1,\ldots,z^{(\alpha)}_d)$, one has  $\mu^{-1}(\cU^{(\alpha)})\cap \tilde{D}_j\neq \varnothing$. We will write $(\cU; z_1,\ldots,z_d)$ instead of $(\mathcal{U}^{(\alpha)}; z^{(\alpha)}_1,\ldots,z^{(\alpha)}_d)$ and $k$ instead of $k_\alpha$ to lighten the notation.  Namely, $\cU\cap D=(z_1\cdots z_k=0)$. Note that $k\geq 1$.

 Pick a point $x\in \tilde{D}_j-\cup_{i\neq j}\tilde{D}_i$ so that  there is an admissible coordinate system $(\cW;x_1,\ldots,x_n)$ with $\mu(\cW)\subset \cU$ and $\cW\cap \tilde{D}=(x_1=0)$. 
 Denote by $ (\mu_1(x),\ldots,\mu_d(x))$ the expression of $\mu$ within these   coordinates. Then
 \[ (\mu_1(x),\ldots,\mu_d(x))=\left(x_1^{n_1}\nu_1(x),\ldots,x_1^{n_k}\nu_{k}(x),\mu_{k+1}(x),\ldots,\mu_{d}(x) \right), 
 \]
 where $\nu_1(x),\ldots,\nu_k(x)$ are holomorphic functions defined on $\mathcal{W}$ so that none of them is  identically equal to zero on   $(x_1=0)$, and $n_p\geq 0$ for $p=1,\ldots,k$. 
 
 We thus can choose a slice $S:=\{(x_1,\ldots,x_d)\mid \{|x_1|\leq \ep, x_2=\zeta_2,\ldots,x_d=\zeta_d  \}\subset \mathcal{W}$ so that  $\nu_i(x)\neq 0$  for any $x\in S$ and any $i=1,\ldots,k$. Let us define a loop $e(\theta)\colon [0,1]\to  \cW^*:=\mathcal{W}-\tilde{D}$ by $  e(\theta):=(\ep e^{2\pi i \theta},\zeta_2,\ldots,\zeta_d) $ which is the generator of $\pi_1(\cW^*, x_0)$, where    $x_0\in \cW^*$ is a point with $\mu(x_0)=y_\alpha\in \cU^*$. By \emph{Cauchy's argument principle}, the winding number of $\mu_p\circ e(\theta)$ around $0$ is $n_p$ for $p=1,\ldots,k$.  Hence by the  diagram
\[
 \begin{tikzcd}
 \pi_1(\mathcal{W}^*, x_0) \arrow[r,"\nu_*"] \arrow[d,"\simeq"] & \pi_1(\mathcal{U}^*, y_\alpha)\arrow[d]\\
 \bZ \arrow[r] & \bZ^k\rlap{\,,}
 \end{tikzcd}
\]
 one has $\nu_*(1)=(n_1,\ldots,n_p)$. 
 
 Pick a path $\tilde{h}\colon [0,1]\to \tilde{U}$ connecting $x$ and $x_0$, which lifts the above path $h_\alpha\colon [0,1]\to U$. Set $\tilde{\gamma}_0\in \pi_1(\tilde{U}, x)$ to be the equivalence class of the loop $\tilde{h}^{-1}\cdot e\cdot \tilde{h}$. Then
\[
 \nu_*(\tilde{\gamma}_0)=\left[h_\alpha^{-1}\cdot \left(e_1^{(\alpha)}\right)^{n_1}\cdots  \left(e_k^{(\alpha)}\right)^{n_k}\cdot h_\alpha\right]=\left(\gamma_1^{(\alpha)}\right)^{n_1}\cdots  \left(\gamma_k^{(\alpha)}\right)^{n_k}
\]
 Therefore,  
 $$
 \left(T_1^{(\alpha)}\right)^{n_1}\cdots  \left(T_k^{(\alpha)}\right)^{n_k}=\rho\left( \nu_*(\tilde{\gamma}_0)\right)\in \tilde{\Gamma}.
 $$
 By \eqref{eq:choice} and \eqref{eq:choice2}, either
 $\rho\left( \nu_*(\tilde{\gamma}_0)\right)=0$, 
   or there is some $i\in \{1,\ldots,k\}$ so that $n_i\geq m$. The first case means that the local monodromy of the pull-back $\bC$-PVHS around $\tilde{D}_j$ is trivial. In the latter case, one has
 $$
 {\ord}_{\tilde{D}_j}(\mu^*D)=\sum_{i=1}^{k}n_i\geq m. $$
The claim is proved.
	\end{proof}

\noindent {\it Step 2}.
Set $D_X\subset D$ to be the sum of  all $\tilde{D}_j$ so that the local monodromy group of the pull-back $\bC$-PVHS around $\tilde{D}_j$ is \emph{not} trivial. Then by the dichotomy in \cref{claim},  $
\mu^*D-mD_X
$ is an effective divisor, and the   pull-back $\bC$-PVHS on $\tilde{U}$ around $\tilde{D}_i$ with $\tilde{D}_i\not\subset D_X$ is trivial. Note that the   pull-back $\bC$-PVHS extends to a $\bC$-PVHS defined over $X-D_X$.

By the second condition in \cref{cond:monodromy},  $(E,\theta)$ is the canonical extension (in the sense of \cref{def:canonical})   of some system of   Hodge bundles $(\tilde{E}=\oplus_{p+q=\ell}\tilde{E}^{p,q},\tilde{\theta},h_{\hod})$ defined over $Y-D$. Hence for any  admissible coordinate $(\cU; z_1,\ldots,z_d)$  and any holomorphic frame $(e_1,\ldots, e_{r'})|_{U}$ for $E^{p,q}$, one has
 \begin{align*}  
   |e_j|_{\hod}\lesssim\frac{1}{\prod_{i=1}^{k}|z_i|^{\ep}} 
 \end{align*}
 for all $\ep>0$. If we take an admissibe coordinate $(\cW;x_1,\ldots,x_d)$ with $\cW\cap \tilde{D}=(x_1\cdots x_c=0)$ and $\mu(\cW)\subset \cU$, one can see that 
 \begin{align*}  
 |\mu^*e_j|_{\mu^*{\hod}}\lesssim\frac{1}{\prod_{i=1}^{c}|x_i|^{\ep\cdot n_i}} \end{align*}
for all $\ep>0$. Here $n_i:={\ord}_{(x_i=0)}(\mu^*(z_1\cdots z_k))$.  
It then follows from the definition of the extension \eqref{def:prolong} that
\begin{align}\label{eq:inclusion}
\mu^*E^{p,q}\subset \dia(\mu^*\tilde{E}^{p,q}). 
\end{align} 
 
 Note that $\mu^*(\tilde{E},\tilde{\theta},h_{\hod})$ is still a system of Hodge bundles over $\tilde{U}$, which  corresponds to the pull-back of the given  $\bC$-PVHS on $U$.  Recall that the   pull-back $\bC$-PVHS extends to a $\bC$-PVHS defined over $X-D_X$. Hence $\mu^*(\tilde{E},\tilde{\theta},h_{\hod})$ extends to   a system of Hodge bundles over $X-D_X$.
 
 We   denote by $(G=\oplus_{p+q=\ell} G^{p,q}, \eta=\oplus_{p+q=\ell}\eta_{p,q})$ the canonical extension (in the sense of \cref{def:canonical}) of $\mu^*(\tilde{E},\tilde{\theta},h_{\hod})$ (which is defined over $X-D_X$) over the log pair $(X,D_X)$, which is thus a system of log Hodge bundles on $(X,D_X)$. In particular, one has
 \begin{align}\label{eq:regular}
 \eta_{p,q}\colon G^{p,q}\lra G^{p-1,q+1}\otimes \Omega_X^1(\log D_X).
 \end{align}
 By \cref{splitting}, one has
\begin{align}\label{equal}
 G^{p,q}=\dia(\mu^*\tilde{E}^{p,q}). 
\end{align}
 Since $L$ is a subsheaf of $E^{p_0,\ell-p_0}$, by \eqref{eq:inclusion} and \eqref{equal},  one has
 $$
 \mu^*L\subset \mu^*E^{p_0,q_0}\subset G^{p_0,q_0}.
 $$ 
 Recall that  $\mu^*D- mD_X$ is an effective divisor and $L-\frac{\ell+1}{m}D$ is a big $\mathbb{Q}$-line bundle.  Write $\tilde{L}:=\mu^*L-\ell D_X$. Then $\tilde{L}$ and $\tilde{L}-D_X$  are both big  line bundles. The above inclusion yields
\begin{align}\label{eq:local}
 \tilde{L}\otimes \cO_X(\ell D_X)\subset G^{p_0,q_0}. 
 \end{align}
 
\medskip

\noindent {\it Step 3}. 
 Now we iterate  $\eta$ $k$ times as in \Cref{sec:iterate} to obtain a morphism 
 \begin{align}\label{iterate2}
  G^{p_0,\ell-p_0}\lra G^{p_0-k,\ell-p_0+k}\otimes  {\Sym}^{k} \Omega_X^1(\log D_X).
  \end{align} 
The inclusion \eqref{eq:local} then induces  a morphism 
\begin{align}\label{eq:eta}
\kappa_k\colon \tilde{L}\otimes \cO_X(\ell D_X)\lra G^{p_0-k,\ell-p_0+k}\otimes  {\Sym}^{k} \Omega_X^1(\log D_X). 
\end{align}
 Write $k_0$ for the largest $k$ so that $\kappa_{k_0}$ is non-trivial. Then $0\leq k_0\leq p_0\leq \ell$. 
Let us denote by $N_p$ the kernel of $\theta_{p,\ell-p}$. Hence $\kappa_{k_0}$ admits a factorization
$$
\kappa_{k_0}\colon \tilde{L}\otimes \cO_X(\ell D_X)\lra N_{p_0-k_0}\otimes  {\Sym}^{k_0} \Omega_X^1(\log D_X).
$$
We first note that  $k_0>0$; or else, there is a morphism from the big line bundle $ \tilde{L}\otimes \cO_X(\ell D_X)$ to $N_{p_0}$, whose dual $N_{p_0}^*$ is weakly positive in the sense of Viehweg by   \cite{Bru17} (see also \cite[Theorem 4.6]{Den20}). 
Hence $\kappa_{k_0}$ induces  
$$
\tilde{L}\lra  N_{p_0-k_0}\otimes  {\Sym}^{k_0} \Omega_X^1(\log D_X)\otimes \cO_X(-\ell D_X)\subset N_{p_0-k_0}\otimes{\Sym}^{k_0} \Omega_X^1 
$$
due to $k_0\leq p_0\leq \ell$. 
In other words, there exists a non-trivial morphism
$$
\tilde{L}\otimes   N^*_{p_0-k_0}\lra {\Sym}^{k_0} \Omega_X^1. 
$$
Recall that $N^*_{p_0-k_0}$ is weakly positive. The torsion-free coherent sheaf $\tilde{L}\otimes   N^*_{p_0-k_0}$ is big in the sense of Viehweg.  Hence there is an $\alpha>0$ so that
$$
{\Sym}^\alpha\left(\tilde{L}\otimes   N^*_{p_0-k_0}\right)\otimes \cO_X(-A)
$$
is generically globally generated for some ample divisor $A$. One thus has a non-trivial morphism
$$
\cO_X(A)\lra {\Sym}^{\alpha k_0}\Omega_X^1.
$$
By a theorem of Campana--P\u{a}un \cite[Corollary 8.7]{CP19}, $X$ is of   general type.

\medskip

\noindent {\it Step 4}. 
Let us prove that $X$ is both pseudo Picard and pseudo Kobayashi hyperbolic. Note that $\kappa_k$ in \eqref{eq:eta} induces a morphism
$$
\tau_k\colon{\Sym}^kT_{X}(-\log D_X)\lra G^{p_0-k,\ell-p_0+k}\otimes  \tilde{L}^{-1}\otimes \cO_X(-\ell D_X)
$$
By \cref{Deng}, we know that $\tau_1$ is   injective on a Zariski open set $\tilde{U}'\subset \tilde{U}$. The morphism    $\tau_k$ induces a morphism
$$
\tilde{\tau}_k\colon {\Sym}^kT_{X}\lra {\Sym}^kT_{X}(-\log D_X)\otimes \cO_X(\ell D_X) \lra G^{p_0-k,\ell-p_0+k}\otimes  \tilde{L}^{-1}
$$
which coincides with $\tau_k$ over $\tilde{U}$.  Hence $\tilde{\tau}_1$ is also injective over $\tilde{U}'$. By \Cref{singular metric}, we can take a singular hermitian metric  $h_{\tilde{L}}$ for $\tilde{L}$ so that  \(h:=h_{\tilde{L}}^{-1}\otimes \tilde{h}_{\hod} \) on \(  G\otimes  \tilde{L}^{-1}\) is locally bounded  on \(Y\) and smooth outside \(D_X\cup \mathbf{B}_+(\tilde{L}-D_X)\), where $\tilde{h}_{\hod}$ is the Hodge metric for the system of Hodge  bundles $(G,\eta)|_{X-D_X}$.  Moreover, \(h\) vanishes  on \(D_X\cup \mathbf{B}_+(\tilde{L}-D_X)\).  This  metric \(h \) on \( G\otimes  \tilde{L}^{-1}\) induces  a Finsler metric \(F_k\) on \( T_X \) defined as follows: for any \(e\in   T_{X,x} \),
$$
F_k(e):=  h\left(\tilde{\tau}_k\left(e^{\otimes k}\right)\right)^{\frac{1}{k}}
$$ 
We   apply the same method as in \Cref{sec:construction} to construct a new Finsler metric $F$   on $T_X$ by taking a convex sum in the form
$$
F:=\sqrt{\sum_{i=1}^{k_0}\alpha_iF_i^2}, 
$$
where \(\alpha_1,\ldots,\alpha_{k_0}\in \mathbb{R}^+ \) are certain constants.  This Finsler metric $F$ on $T_X$  is   positive definite over $\tilde{U}^\circ:=\tilde{U}'-\mathbf{B}_+(\tilde{L}-D_X)$ as $\tilde{\tau}_1$ is   injective over $\tilde{U}'$ and $h$ is smooth on $\tilde{U}- \mathbf{B}_+(\tilde{L}-D_X)$. Set $Z:=X\setminus \tilde{U}^\circ$, which is a proper Zariski closed subvariety of $X$. By \cref{curvature estimate} one can choose  \(\alpha_1,\ldots,\alpha_{k_0}\in \mathbb{R}^+ \) properly  so that for  any $\gamma\colon C\to X$ with $C$ an open subset of $\bC$ and $\gamma(C)\cap \tilde{U}^\circ\neq \varnothing$,   one has
\begin{align} \label{eq:curvature}
\hess \log|\gamma'(t)|_{F}^2\geq  \gamma^*\omega 
\end{align} 
for some fixed smooth K\"ahler form $\omega$ on $X$. Indeed, it follows from the proof of \cref{curvature estimate} that there is an  open subset $C^\circ$ of $C$ whose complement is a discrete set such that \eqref{eq:curvature} holds over $C^\circ$. By \cref{def:Finsler}, $|\gamma'(t)|_F^2$ is continuous and locally bounded from above over   $C$,  and by the extension theorem of  subharmonic functions, \eqref{eq:curvature}  holds over the whole unit disk $C$.   Applying  \cref{thm:criteria} to \eqref{eq:curvature}, we conclude that $X$ is Picard hyperbolic modulo $Z$.  Hence \cref{pseudo Picard} follows. 

Let $C$ be an irreducible compact curve in $X$ not contained in $Z$. Write $h_{\tilde{C}}$ the induced singular hermitian metric for $T_{\tilde{C}}$ by $F$, where $\tilde{C}$ is the normalization of $C$.   Then by \eqref{eq:curvature}, one has
$$
2g(\tilde{C})-2=-\sn\Theta_{h_{\tilde{C}}}(T_{\tilde{C}})\geq \deg_{\omega}(C).
$$
This proves \cref{pseudo algebraic}.

By \cref{def:Finsler} again, there is an $\varepsilon>0$ so that $\omega\geq \ep F^2$. Hence \eqref{eq:curvature} implies that
$$
	\frac{\d^2  \log|\gamma'(t)|_{F}^2}{\d t\d \bar{t}} \geq \ep |\gamma'(t)|_{F}^2
$$
for any $\gamma\colon \Delta\to X$ with $\gamma(\Delta)\cap \tilde{U}^\circ\neq \varnothing$. In other words, the  holomorphic sectional curvature of $F$ is bounded from above by the negative constant $- \ep$ (see \cite[Theorem 2.3.5]{Kob98}). By the Ahlfors--Schwarz lemma, we conclude that  $X$ is Kobayashi hyperbolic modulo $Z$ (see \cite[Lemma 2.4]{Den18}). This proves \cref{pseudo Kobayashi}. 
 The theorem is proved.
\end{proof}  
During the above proof, we indeed obtained the following result.

\begin{thm}\label{thm:new}
	Let $(Y,D)$ be a compact K\"ahler log pair, and let  $(E,\theta)=(\oplus_{p+q=\ell}E^{p,q},\oplus_{p+q=\ell}\theta_{p,q})$ be a system of log Hodge bundles  on  $(Y,D)$   satisfying the following conditions: 
	\begin{enumerate}
		\item The pair  $(E,\theta) $ is the canonical extension of some system of Hodge bundles over $Y-D$ of weight $\ell$.
		\item  There is a big   line bundle $L$ over $Y$ such that $L\otimes \cO_X(\ell D)\subset E^{p_0,q_0}$ for some $p_0+q_0=\ell$.
		\item The line bundle $L\otimes \cO_X(-D)$ is  still big.
	\end{enumerate} 
Then  there is a proper Zariski closed subset $Z\subsetneq Y$ so that 
	\begin{thmlist}
		 	\item $Y$ is  Kobayashi hyperbolic modulo $Z$;
		\item $Y$ is  Picard hyperbolic modulo $Z$;
		\item $Y$ is algebraically hyperbolic modulo $Z$.
		\item  $Y$ is of general type; 
		\end{thmlist}
\end{thm}

Now we are able to prove \cref{thmx:refined}.

\begin{proof}[Proof of \cref{thmx:refined}]
 By  Steps  1--3 in the proof of \cref{strong},  we can construct a projective log pair $(X,D)$ and a log morphism $\mu\colon (X,\tilde{D})\to(Y,D)$ which is a finite \'etale cover over $U$. Over $(X,D)$, there is a system of log Hodge bundles  $(G=\oplus_{p+q=\ell}G^{p,q},\eta=\oplus_{p+q=\ell}\eta_{p,q})$  satisfying the following properties:
\begin{enumerate}
	\item The pair $(G,\eta)$ is the canonical extension of some system of Hodge bundle on $X-D_X$,  where $D_X$ is a reduced simple normal crossing divisor supported on $D$.
	\item \label{2} There is a big   line bundle $\tilde{L}$ on $X$ so that $\tilde{L}\otimes \cO(-D_X)$ is also big,
	\item \label{3}  There is an inclusion $\tilde{L}\otimes \cO_X(\ell D_X)\subset G^{p_0,\ell-p_0}$ for some $p_0> 0$. 
\end{enumerate}  
Since the period map has zero-dimensional fibers,  by \cref{augmented base} and the construction of $\tilde{L}$ at the beginning of the proof of \cref{strong}, we moreover have that 
\begin{enumerate}[resume]
	\item \label{aug} the augmented base locus satisfies $\mathbf{B}_+(\tilde{L})\subset \tilde{D}$.
\end{enumerate}
 
Let $\tilde{Z}$ be any irreducible Zariski closed subvariety of $X$ of positive dimension which is not contained in~$\tilde{D}$. Take a resolution of
 singularities   $g\colon Z\to \tilde{Z}$  so that $D_Z:=\nu^{-1}(D_X)$ is simple normal crossing.  Then $g\colon (Z,D_Z)\to (X,D_X)$ is a log morphism which is generically finite.  
 
By \cref{aug}, we can see that  $L_Z:=g^*\tilde{L}$ is big. Since $g^*D_X-D_Z$ is an effective divisor,  $L_Z\otimes \cO_Z(-D_Z)$ is also big. By \cref{3}, one has
 $$
L_Z\otimes \cO_Z(\ell D_Z)\subset g^*\left(\tilde{L}\otimes \cO_X(\ell D_X)\right)\subset g^*G^{p_0,\ell-p_0}.
 $$
 For the $\bC$-PVHS corresponding to $(G,\eta)|_{X-D_X}$, we pull it back to $Z-D_Z$ via $g$ and denote by $(\tilde{E},\tilde{\theta})$ the induced system of Hodge bundles on $Z-D_Z$. Let $(E=\oplus_{p+q=\ell}E^{p,q}, \theta)$ be the canonical extension of such a system of Hodge bundles. In the same vein as the proof of \eqref{eq:inclusion}, one has
 $$
 g^*G^{p_0,\ell-p_0}\subset E^{p_0,\ell-p_0}.
 $$
 In summary, we construct a system of log Hodge bundles $(E=\oplus_{p+q=\ell}E^{p,q}, \theta)$ on $(Z,D_Z)$ satisfying the two conditions in \cref{thm:new}.    By \cref{thm:new},  $Z$ is of general type.  We have proved \cref{general type2}.
  
Let us prove \cref{Picard}. For any $\tilde{\gamma}\colon \Delta^*\to X$ whose image is not contained in $\tilde{D}$, let $\tilde{Z}$ be its Zariski closure. Take a desingularization $\nu\colon Z\to \tilde{Z}$ as above, and let $\gamma\colon\Delta^*\to Z$ be the lift of $\gamma$. By the above argument and \cref{thm:new},    $\gamma$ extends to a holomorphic map $\overline{\gamma}\colon\Delta\to Z$. Therefore,  $\nu\circ\overline{\gamma}$ extends $\tilde{\gamma}$. We proved \cref{Picard}. It is   easy to see that \cref{Picard} implies \cref{Brody}.

The proof of \cref{Algebraic} is exactly the same as that of \cref{main}. We will not repeat the arguments and leave the proof to the interested readers.
\end{proof}

We now show how to deduce \cref{Picard symmetric} from \cref{thmx:refined}. 

\begin{proof}[Proof of \cref{Picard symmetric}]
 By the work of Baily--Borel and Mok, we know that $U$ is  quasi-projective.  By the work of Deligne,  $U$ admits a  $\bC$-PVHS whose period map is immersive everywhere (see, \textit{e.g.}, \cite[Theorem 7.10]{Mil13}).  The corollary immediately follows from \cref{Picard}.
	\end{proof}

\begin{rem}\label{rem:compare}
\cref{Picard symmetric} unifies the previous result by Nadel who proved that $X$ is Brody hyperbolic modulo $X-\tilde{U}$. Applying \cref{general type2},  it also re-proves    theorems by Brunebarbe \cite{Bru16} and Cadorel \cite{Cad18}:  any positive-dimensional irreducible subvariety of $X$ not contained in $X-\tilde{U}$ is of general type. However, since our proof does not rely on  special properties of bounded symmetric domains (we use neither Mumford's work on toroidal compactifications nor the existence of variations of Hodge structures of Calabi-Yau type over quotients of bounded symmetric domains by arithmetic groups), we certainly loose the effectivity result regarding the level structures of the \'etale coverings, which is also a main result in  \cite{Nad89,Bru16,Cad18}.
\end{rem}

\end{document}